\documentclass{ams}
\usepackage{bbm}
\usepackage{mathtools}
\usepackage{chemarrow}
\usepackage{amsfonts}
\usepackage{mathrsfs}
\usepackage{stmaryrd}
\usepackage{amssymb}

\newcommand{\dn}{\mathord{\downarrow}\hspace{0.05em}}

\newtheorem{theorem}{Theorem}[section]
\newtheorem{lemma}[theorem]{Lemma}
\newtheorem{claim}[theorem]{Claim}
\newtheorem{thm1.6}[theorem]{Theorem \ref{thm.1.7}}
\newtheorem{thm1.7}[theorem]{Theorem \ref{thm.1.8}}

\theoremstyle{definition}
\newtheorem{definition}[theorem]{Definition}
\newtheorem{example}[theorem]{Example}

\theoremstyle{remark}
\newtheorem{remark}[theorem]{Remark}

\numberwithin{equation}{section}

\newcommand{\Fraisse}{Fra\"{\i}ss\'{e}}

\begin{document}

\title{Big Ramsey degrees in universal inverse limit structures}

\author{Natasha Dobrinen}
\address{Department of Mathematics, University of Denver, Denver 80208, USA}
\email{natasha.dobrinen@du.edu}
\thanks{Dobrinen was supported by  National Science Foundation Grants DMS-1600781 and  DMS-1901753.
Wang was supported by  the National Natural Science Foundation of
China (Grant nos. 11871320 and 11531009), Innovation Capability Support Program of Shaanxi (Program no. 2020KJXX-066)  and  the
Fundamental Research Funds for the Central Universities (Grant no. GK202102011).}

\author{Kaiyun Wang}
\address{School of Mathematics and Statistics, Shaanxi Normal University, Xi'an 710119, P.R. China}
\email{wangkaiyun@snnu.edu.cn}


\date{}

\thanks{}
\begin{abstract}
We build a collection of  topological Ramsey spaces of trees giving rise to universal inverse limit structures,
extending Zheng's work for the profinite graph to the setting of
Fra\"{\i}ss\'{e} classes of finite ordered binary relational structures with the Ramsey property.
This work is
based on the Halpern-L\"{a}uchli  theorem, but different from the Milliken space of strong
subtrees.
Based on these topological Ramsey spaces and the work of Huber-Geschke-Kojman on
inverse limits of finite ordered graphs, we prove that for each such Fra\"{\i}ss\'{e} class,
its universal inverse limit structure  has finite big Ramsey degrees under  finite Baire-measurable colorings.
For such \Fraisse\ classes  satisfying free amalgamation  as well as
 finite ordered tournaments and finite partial orders with a linear extension,  we characterize  the exact big Ramsey degrees.
\end{abstract}

\maketitle

\vspace*{0.2cm}

\section{Introduction}

\vspace*{0.4cm}

Structural Ramsey theory originated at the beginning of the  1970's in a series of
papers (see \cite{JN95}). Given structures $\pmb{A}$ and $\pmb{B}$, let $\big(\mathop{}_{\pmb{A}}^{\pmb{B}}\big)$ denote the set of
 all copies of $\pmb{A}$ in  $\pmb{B}$.
 We write  $\pmb{C}\longrightarrow (\pmb{B})^{\pmb{A}}_{l,m}$ to denote the following property:
 For every finite coloring $c:  \big(\mathop{}_{\pmb{A}}^{\pmb{C}}\big)
\longrightarrow l$, there is $\pmb{B}^\prime\in \big(\mathop{}_{\pmb{B}}^{\pmb{C}}\big)$ such that $c$  takes no  more than $m$  colors on
$\big(\mathop{}_{\pmb{A}}^{\pmb{B^\prime}}\big)$.
Let $\mathcal{K}$ be  a class
 of structures.  The (small) Ramsey degree of $\pmb{A}$ in $\mathcal{K}$ is the  smallest positive integer $m$, if it exists,  such
that for every $\pmb{B}\in \mathcal{K}$ and every positive integer $l\geq 2$ there exists $\pmb{C}\in \mathcal{K}$ such that
$\pmb{C}\longrightarrow (\pmb{B})^{\pmb{A}}_{l,m}$.
The class $\mathcal{K}$ is said to be a {\em Ramsey class} if the Ramsey degree of every $A\in \mathcal{K}$ is
1. Ramsey classes are the main topic of interest of  structural Ramsey theory.
Many Ramsey classes are known.
Examples relevant for our
presentation include the classes of finite ordered graphs, finite ordered $k$-clique free graphs with
$k\geq 3$, finite ordered oriented graphs,  finite ordered tournaments, and finite partial orders with a linear extension.

Given an infinite structure $\pmb{S}$ and a finite substructure $\pmb{A}$,  the {\em big Ramsey degree}
of $\pmb{A}$ in $\pmb{S}$ is the  smallest positive integer $m$, if it exists, such that $\pmb{S}\longrightarrow (\pmb{S})^{\pmb{A}}_{l,m}$ for every $l\geq 2$.
Research on big Ramsey degrees has gained recent momentum
due to  the  seminal paper of  Kechris, Pestov, and Todor\v{c}evi\'{c} in \cite{AK05}, and the    results
 by Zucker in \cite{AZ19}
connecting big Ramsey degrees for countable structures with topological dynamics, answering a question in \cite{AK05}.

The history of big Ramsey degrees for countably infinite  structures has its beginnings in an example of Sierpi\'{n}ski, who constructed a 2-coloring of pairs of rationals such that every subset forming a dense linear order retains both colors.
Later, Galvin proved that for every finite coloring of pairs of rationals, there is a subset forming a dense linear order on which the coloring takes no more than two colors, thus proving that the big Ramsey degree for pairs of rationals is exactly two.
This line of work has developed over the decades, notably with Laver proving upper bounds for all finite sets of rationals, and culminating in  Devlin's calculations of the exact big Ramsey degrees for finite sets of rationals in
\cite{DD79}.

The area of big Ramsey degrees on countably infinite structures  has seen considerable growth
in the
 past two decades,
beginning notably with
 Sauer's proof in \cite{NS06} that every finite graph has finite big Ramsey degree in the Rado graph, which is  the
Fra\"{\i}ss\'{e} limit of the class of all the finite graphs, and the immediately following result of
Laflamme, Sauer and Vuksanovic  in \cite{CL06} characterizing  the exact big Ramsey degrees  of the Rado graph.
Other recent work on big Ramsey degrees of countable structures include
ultrametric spaces (Nguyen Van Th\'{e}, \cite{LN08}),
the dense local order
(Laflamme, Nguyen Van Th\'{e} and Sauer, \cite{CL10}),
the ultrahomogeneous $k$-clique free graphs
(Dobrinen,  \cite{ND17, ND19}),
and, very recently,   the following:
\cite{Balkoproofs}, \cite{Balko7},
\cite{CDP}, \cite{Hubicka20},
\cite{MasulovicFBRD18},
\cite{MasulovicRDBVS20},
\cite{ZuckerForb}.
For more background in this area, we refer the reader to the excellent Habilitation of Nguyen Van Th\'{e}
\cite{NVT}
and a more recent expository paper of the first author \cite{DobIfCoLog}.

Results on big Ramsey degrees for uncountable structures are  even more sparse than for countable structures.
Ramsey theorems for perfect sets  mark a beginning of this line of inquiry, and most of these theorems have at their core either the Milliken theorem (\cite{KR81}),
or the Halpern-L\"{a}uchli theorem (\cite {JH66}) on which Milliken's theorem is based.
  For example, Blass proved in \cite{AB81} the following partition theorem for perfect sets of $\mathbb{R}$, which was  conjectured by Galvin (see \cite{FG68}), who proved it for $n\leq 3$.

\begin{theorem}[Blass \cite{AB81}]\label{thm.Blass} For every perfect subset $P$ of $\mathbb{R}$ and every finite continuous
coloring of $[P]
^n$, there is a perfect set $Q\subseteq P$ such that $[Q]^n$ has at most $(n- 1)!$
colors.
\end{theorem}

In the proof of  this theorem, Blass defined patterns for finite subsets
of a perfect tree $T$ such that for every finite continuous coloring of finite subsets of the
nodes in $T$, one can make all subsets with a fixed pattern monochromatic by going to
a perfect subtree. Todor\v{c}evi\'{c} (see \cite{ST10}, Corollary 6.47) provided a simpler proof of Blass'
theorem using the Milliken space, as the perfect trees in Blass' argument can be
replaced by strong subtrees.

Given a set $X$, a subset $Y$ is called an $n$-subset of $X$ if $Y$ is a subset of $X$ of size $n$. Let
$[X]^n=\{Y\subseteq X : |Y|=n\}$ be the set of all $n$-subsets of $X$.  For a graph $G$, let $V(G)$
denote its vertex set and $E(G)\subseteq [V(G)]^2$ denote its edge relation, that is, $E(G)$ is an irreflexive and symmetric binary relation.  An inverse limit of finite
ordered graphs is called {\em universal} if every inverse limit of finite
ordered graphs order-embeds
continuously into it. Geschke (see \cite{SG13}) proved the existence of a universal inverse limit graph. Moreover, Huber-Geschke-Kojman (see \cite{SH19}) gave the definition of  a universal   inverse limit graph with no mention of an inverse system.

\begin{definition}[Huber-Geschke-Kojman \cite{SH19}]
 A {\em universal inverse limit} of finite
ordered graphs
 is a triple $G=\langle V, E, <\rangle$, such
that the following conditions hold.
\begin{enumerate}
\item
 $V$ is a compact subset of $\mathbb{R}\backslash\mathbb{Q}$, $E\subseteq [V]
^2$, and $<$ is the
restriction of the standard order on $\mathbb{R}$ to $V$.
\item
 (Modular profiniteness)  For every pair of distinct vertices $u, v\in  V$, there is a partition of $V$ to finitely
many closed intervals such that
\begin{enumerate}
\item[(a)]
 $u, v$ belong to different intervals from the partition;

\item[(b)]
 for every interval $I$ in the partition, for all $x\in V\backslash I$ and for all $y, z\in I$,
$(x, y)\in E$ if and only if $(x, z)\in E$.
\end{enumerate}
\item[(3)]
(Universality) Every nonempty open interval of $V$ contains induced copies of all finite ordered
 graphs.
 \end{enumerate}
\end{definition}

Based on  the way Blass proved Theorem \ref{thm.Blass} by partitioning finite subsets into patterns
in \cite{AB81}, Huber, Geschke and Kojman proved in \cite{SH19} the following  partition theorem for universal inverse limits of finite ordered graphs  by partitioning the
isomorphism class of finite ordered graph $H$ into $T(H)$ many  strong isomorphism
classes,  called types.
This theorem tells us that the universal inverse limit graphs have
finite big Ramsey degrees under  finite Baire-measurable colorings.

\begin{theorem}[Huber-Geschke-Kojman \cite{SH19}]\label{thm.HGK} For every finite ordered graph $H$ there is
$T(H)< \omega$ such that for every  universal inverse limit  graph $G$, and
for every finite Baire-measurable coloring of the set
$\big(\mathop{}_{H}^{G}\big)$ of all copies of $H$ in $G$, there is
a closed copy $G^\prime$ of $G$ in $G$ such that the set $\big(\mathop{}_{H}^{G^\prime}\big)$ of all copies of $H$ in $G^\prime$ has at most $T(H)$
colours.
\end{theorem}

The following notation will be used throughout.
The set of natural numbers, $\{0,1,2,3,\dots\}$, will be denoted by $\omega$.
Let $\omega^{<\omega}$ be the set of all finite sequences of natural numbers. Let $\subseteq$ denote
the initial segment relation. For an element $s\in  \omega^{<\omega}$, let $|s|$ denote the length of $s$. We call a downward closed subset $T$ of $\omega^{<\omega}$ a tree, ordered by $\subseteq$. Every element
$t$ of a tree $T$ is called a node. Given a tree $T$, let $[T]$ be the set of all infinite branches of
$T$, i.e.,
$[T]=\{x\in \omega^{\omega} : (\forall \ n< \omega) \ x\upharpoonright n\in  T\}$, where $x\upharpoonright n$ is its initial
segment of length $n$.  $T^\prime$ is called  a subtree of $T$ if $T^\prime\subseteq  T$ and $T^\prime$ is a tree. For a tree $T$ and $t\in T$,
$s$ is called an immediate successor of
$t$ if $s$ is a minimal element of $T$ above $t$. The set of immediate successors of
$t$ in $T$ is denoted by $\mbox{succ}_T(t)$. Let $T_t$ be the set of all nodes in $T$ comparable to $t$, i.e.,  $T_t =\{s\in  T : t\subseteq s \vee s\subseteq
t\}$. For $n\in \omega$, we let $T(n)=\{t\in T : |t|=n\}$.

In order to state the results of Huber-Geschke-Kojman and of Zheng, we need to introduce the following notation and structures.
Let $\mathrm{R}$ denote the Rado graph, i.e., the unique (up to isomorphism) countable universal  homogeneous
graph. We assume that the set of vertices of $\mathrm{R}$ is just the set $\omega$
of natural numbers.  For $n\in \omega$, let $\mathrm{R}_n$ be
the induced   subgraph of $\mathrm{R}$ on $\{0,\dots,n\}$.

\begin{definition}[\cite{SH19}] Let $T_{\max}\subseteq \omega^{<\omega}$ be the nonempty tree such that for each $t\in  T_{\max}$,
$$\mbox{succ}_{T_{\max}}(t)=\{t^\frown\langle0\rangle, t^\frown\langle1\rangle, \dots,  t^\frown\langle|t|\rangle\}.$$
For $t\in T_{\max}$, we define $G_t$ to be the ordered  graph on the vertex set $\mbox{succ}_{T_{\max}}(t)$ with
lexicographical ordering, such that $G_t$ is isomorphic to  $\mathrm{R}_{|t|}$.
\end{definition}

Note that $[T_{\max}]$ is a subset of $\omega^{\omega}$.
Given  $x, y\in  [T_{\max}]$ with $x\neq y$, let $x\cap y\in \omega^{<\omega}$ be the
common initial segment of $x$ and $y$, i.e. $x\cap y= x\upharpoonright \min\{n : x(n)\neq y(n)\}$. The tree $T_{\max}$ and the ordered graphs $G_t \ (t\in T_{\max})$ induce an ordered  graph $G_{\max}$ on the
vertex set $[T_{\max}]$, ordered lexicographically, with the edge relation defined as follows.
For $x, y\in [T_{\max}], (x, y)\in  E(G_{\max})$ if and only if
$(x\upharpoonright (|x\cap y|+1), y\upharpoonright (|x\cap y|+1))\in E(G_{x\cap y})$.  Suppose that $T$ is  a  subtree of $T_{\max}$ and $t\in T$. Let
$G^T_t$ denote the induced  subgraph of $G_t$ on the vertex set $\mbox{succ}_{T}(t)$. We define $G(T)$ to
be the induced   subgraph of $G_{\max}$ on $[T]$.
A subtree $T$ of $T_{\max}$ is called a {\em $G_{\max}$-tree} if for every finite ordered  graph $H$  and every $t\in T$, there is $s\in T$
with $t\subseteq s$ such that $H$ embeds into $G^T_s$.
In particular, $T_{\max}$ is a $G_{\max}$-tree.

Let $(\mathcal{R}, \leq, r)$ be a triple satisfying the  following: $\mathcal{R}$ is a nonempty set, $\leq$ is a quasi-ordering on $\mathcal{R}$, and $r : \mathcal{R}\times \omega\longrightarrow \mathcal{AR}$ is a
mapping giving us the sequence ($r_n(\cdot)=r(\cdot,n)$) of approximation mappings, where
$$\mathcal{AR}=\{r_n(A):A\in\mathcal{R}\mathrm{\ and\ } n\in\omega\}.
$$
For $a\in \mathcal{AR}$ and $A\in  \mathcal{R}$,
$$[a, A]=\{B\in  \mathcal{R} : (B\leq A)\wedge (\exists n)(r_n(A)=a)\}.$$
The topology on $\mathcal{R}$ is given by the basic open sets $[a, A]$. This topology is called
the {\em Ellentuck topology} on $\mathcal{R}$. Given the Ellentuck topology
on $\mathcal{R}$, the notions of nowhere dense, and hence of meager are defined in the natural
way. Thus, we may say that a subset $\mathcal{X}$ of $\mathcal{R}$ has the {\em property of Baire} iff $\mathcal{X}=\mathcal{O}\triangle \mathcal{M}$
for some Ellentuck open set $\mathcal{O}\subseteq \mathcal{R}$ and Ellentuck meager set $\mathcal{M }\subseteq \mathcal{R}$.

\begin{definition}[\cite{ST10}]  A subset $\mathcal{X}$ of $\mathcal{R}$ is {\em Ramsey} if for every $\emptyset\neq [a, A]$, there is a
$B\in [a, A]$ such that $[a, B]\subseteq \mathcal{X}$ or $[a, B]\cap \mathcal{X}=\emptyset$. $\mathcal{X}\subseteq \mathcal{R}$ is {\em Ramsey null} if for
every $\emptyset\neq [a, A]$, there is a
$B\in [a, A]$ such that $[a, B]\cap \mathcal{X}=\emptyset$.

A triple $(\mathcal{R},\leq,r)$ is a {\em topological Ramsey space} if every property of Baire subset
of $\mathcal{R}$ is Ramsey and if every meager subset of $\mathcal{R}$ is Ramsey null.
\end{definition}

In \cite{YY18}, Zheng constructed a collection of topological Ramsey spaces of trees. For  each type $\tau$ of finite ordered graphs, the space $(\mathcal{G}_\infty(\tau), \leq, r)$ consists of $G_{\max}$-trees of a
particular shape.  The new spaces $\mathcal{G}_\infty(\tau)$ not only depend  on  the fact that the class of finite ordered graphs is the Ramsey class, but also, similarly to the Milliken space, are based
on the Halpern-L\"{a}uchli  theorem.
Moreover, she presented an application of the topological Ramsey spaces $\mathcal{G}_\infty(\tau)$ to inverse limit graph theory.
Similarly to how Todor\v{c}evi\'{c}  proved
 Blass' Theorem \ref{thm.Blass}, Zheng used  the new spaces $\mathcal{G}_\infty(\tau)$ to prove
the following Theorem \ref{thm.Zheng3.1} (Theorem 3.1 in \cite{SH19}),
 which is a key step to show the above Theorem
 \ref{thm.HGK}
  in \cite{SH19}.

\begin{theorem}[Theorem 3.1 in \cite{SH19}]\label{thm.Zheng3.1}
Let $T$ be an arbitrary $G_{\max}$-tree. For every type $\tau$ of a
finite induced subgraph of $G_{\max}$, and for every
continuous coloring $c :
\big(\mathop{}_{ \ \ \tau}^{G(T)}\big)\longrightarrow 2$,
there is a $G_{\max}$-subtree $S$ of $T$ such that c is constant on $\big(\mathop{}_{ \ \ \tau}^{G(S)}\big)$.
\end{theorem}

In this paper, we extend Zheng's  methods to build  a collection of  topological Ramsey spaces of trees in the setting of
Fra\"{\i}ss\'{e} classes of finite ordered structures with finitely many binary relations satisfying the Ramsey property.
 Based on these topological Ramsey spaces and the work of Huber-Geschke-Kojman on
   inverse limits of  finite ordered graphs, we prove the following theorem.
Here, $\pmb{F_{\max}}$ is a universal limit structure encoded
in a particular way  on
 the set of   infinite branches of a  certain  finitely branching tree $T_{\max}$
(see Definitions \ref{defn.T_max} and \ref{defn.F_max}).

\begin{theorem}\label{thm.1.7}
 Let $\mathcal{K}$ be a Fra\"{\i}ss\'{e} class, in a finite   signature,
 of finite ordered binary relational structures with the Ramsey property. For every $\pmb{H}\in \mathcal{K}$,  there is a finite number $T(\pmb{H}, \pmb{F_{\max}})$ such that for every universal  inverse limit structure
$\pmb{G}$, for every finite Baire-measurable coloring of the set $\big(\mathop{}_{\pmb{H}}^{\pmb{G}}\big)$
of all copies of $\pmb{H}$ in $\pmb{G}$, there is a closed
copy $\pmb{G^\prime}$ of $\pmb{G}$ contained in $\pmb{G}$ such that  the set $\big(\mathop{}_{\pmb{H}}^{\pmb{G^\prime}}\big)$
of all copies of $\pmb{H}$ in $\pmb{G^\prime}$ has no more than $T(\pmb{H}, \pmb{F_{\max}})$ colors.
\end{theorem}

This means that  for each such Fra\"{\i}ss\'{e} class,
its universal inverse limit structures  have finite big Ramsey degrees under  finite Baire-measurable colorings.
For the following classes, we characterize the big Ramsey degrees in terms of  types.

\begin{theorem}\label{thm.1.8}
Let $\mathcal{K}$ be a \Fraisse\ class in a finite binary relational signature  such that one of the following hold:
\begin{enumerate}
\item
$\mathcal{K}$ is an ordered expansion of a free amalgamation class;
\item
$\mathcal{K}$ is the class of finite ordered tournaments;
\item
$\mathcal{K}$ is the class of finite partial orders with a linear extension.
\end{enumerate}
Let  $\pmb{G}$ be
 a  universal inverse limit structure  for  $\mathcal{K}$ contained in
 $\pmb{F_{\max}}$.
Then for each $\pmb{H}\in \mathcal{K}$, each type representing $\pmb{H}$  in $\pmb{G}$ persists in each closed subcopy  of $\pmb{G}$.
It follows that
 the  big Ramsey degree $T(\pmb{H}, \pmb{F_{\max}})$
for finite
Baire-measurable
colorings of
$\big(\mathop{}_{ \ \ \pmb{H}}^{\pmb{F_{\max}}}\big)$
is exactly the number of types in $T_{\max}$ representing a copy of $\pmb{H}$.
\end{theorem}




\section{Ordered binary relational Fra\"{\i}ss\'{e} classes and $\pmb{F_{\max}}$-trees}


Let us first review some basic facts of the Fra\"{\i}ss\'{e} theory  for finite ordered binary relational  structures which are
necessary to this paper.
More general background on Fra\"{\i}ss\'{e} theory can be found in \cite{AK05}.

We shall call $L=\{<, R_0,\dots,R_{k-1}\}$ an {\em ordered binary relational signature} if it consists
of the order relation symbol $<$ and finitely many  binary relation symbols $R_{\ell}$, $\ell<k$ for some $k<\omega$.
 A structure for $L$ is of the form $\pmb{A}=\langle A, <^{\pmb{A}}, R_0^{\pmb{A}},\dots,R_{k-1}^{\pmb{A}}\rangle$, where $A\neq \emptyset$  is the universe of $\pmb{A},\ <^{\pmb{A}}$ is a linear
ordering of $A$,  and  each  $R_{\ell}^{\pmb{A}}\subseteq A\times A$.
 An embedding between structures $\pmb{A}, \pmb{B}$ for $L$ is an injection
$\pi : A\longrightarrow B$ such that for any two $a, a^\prime\in A,\ a<^{\pmb{A}} a^\prime\Longleftrightarrow
\pi(a)<^{\pmb{B}} \pi(a^\prime)$ and
for each $\ell<k$,
$(a_1, a_2)\in R_\ell^{\pmb{A}}\Longleftrightarrow (\pi(a_1),\pi(a_2))\in R_\ell^{\pmb{B}}$.
If $\pi$ is the identity, we say
that $\pmb{A}$ is a substructure of $\pmb{B}$. An isomorphism is an onto embedding. We
write $\pmb{A}\leq \pmb{B}$ if $\pmb{A}$ can be embedded in $\pmb{B}$ and $\pmb{A}\cong \pmb{B}$ if $\pmb{A}$  is isomorphic to $\pmb{B}$.

A class $\mathcal{K}$ of finite structures is called {\em hereditary} if $\pmb{A}\leq \pmb{B}\in  \mathcal{K}$ implies $\pmb{A}\in  \mathcal{K}$.
 It satisfies the {\em joint embedding property} if for any $\pmb{A},\ \pmb{B}\in  \mathcal{K}$, there is
$\pmb{C}\in  \mathcal{K}$ with $\pmb{A}\leq \pmb{C}$ and $\pmb{B}\leq \pmb{C}$.  We say that $\mathcal{K}$ satisfies the {\em amalgamation property}
if for any embeddings $f : \pmb{A}\longrightarrow \pmb{B},\ g : \pmb{A}\longrightarrow \pmb{C}$ with $\pmb{A},\ \pmb{B},\ \pmb{C}\in  \mathcal{K}$, there is
$\pmb{D}\in  \mathcal{K}$ and embeddings $r : \pmb{B}\longrightarrow \pmb{D}$ and $s : \pmb{C}\longrightarrow \pmb{D}$, such that $r\circ f= s\circ g$.
A class  of finite structures  $\mathcal{K}$ is called a {\em Fra\"{\i}ss\'{e} class}   if it is hereditary, satisfies joint embedding and
amalgamation, contains only countably many structures, up to isomorphism,
and contains structures of arbitrarily large (finite) cardinality.
A Fra\"{\i}ss\'{e} class satisfies the {\em free amalgamation property}  (or {\em has free amalgamation})
if $\pmb{D}$, $r$, and $s$ in the amalgamation property can be chosen
so that
$r[B]\cap s[C]=r\circ f[A]=s\circ g[A]$, and
$\pmb{D}$ has no additional relations on its universe other than those inherited from $\pmb{B}$ and $\pmb{C}$.

Let $\pmb{A}$ be a structure for $L$. For each $X\subseteq A$, there is a smallest substructure containing $X$, called the substructure generated by $X$. A substructure is called finitely generated if it is generated by a finite set. A structure is locally finite if all its finitely generated substructures are finite. The age of $\pmb{A}$, Age$(\pmb{A})$  is the class of all finitely  generated structures in $L$ which can be embedded in $\pmb{A}$.
We call $\pmb{A}$ {\em  ultrahomogeneous} if every
isomorphism between finitely generated  substructures of $\pmb{A}$  can be extended to an automorphism
of $\pmb{A}$.
A locally finite, countably infinite, ultrahomogeneous
structure  is called a {\em Fra\"{\i}ss\'{e} structure}.

There is a canonical one-to-one correspondence between
Fra\"{\i}ss\'{e} classes of  finite   structures and Fra\"{\i}ss\'{e} structures,
discovered by Fra\"{\i}ss\'{e}.
If $\pmb{A}$ is a Fra\"{\i}ss\'{e} structure, then Age$(\pmb{A})$ is a
Fra\"{\i}ss\'{e} class of finite   structures. Conversely, if $\mathcal{K}$ is a Fra\"{\i}ss\'{e} class of  relational structures, then there is a unique Fra\"{\i}ss\'{e}
structure, called  the Fra\"{\i}ss\'{e} limit of $\mathcal{K}$, denoted by Flim$(\mathcal{K})$, whose age is exactly
$\mathcal{K}$.

\begin{definition}
Let  $\mathcal{K}$ be a Fra\"{\i}ss\'{e} class  of finite ordered binary relational structures. We say that $\mathcal{K}$ satisfies the {\em Ramsey property} if $\mathcal{K}$ is a Ramsey class.
That is, for each
$\pmb{A},\pmb{B}\in\mathcal{K}$
such that $\pmb{A}\le \pmb{B}$ and
for every positive integer $l\geq 2$, there exists $\pmb{C}\in \mathcal{K}$ such that
$\pmb{C}\longrightarrow (\pmb{B})^{\pmb{A}}_{l}$.
\end{definition}

Given an ordered   binary relational signature
$L=\{<,R_0,\dots, R_{k-1}\}$,
let $L^-$  denote $\{R_0,\dots, R_{k-1}\}$.
An $L^-$-structure  $\pmb{A}$ is called {\em irreducible} if for any two elements $x,y \in A$, there is some relation $R\in L^-$ such that either $R^{\pmb{A}}(x,y)$
or  $R^{\pmb{A}}(y,x)$
 holds.
Given a set $\mathcal{F}$ of  finite $L^-$-structures, let Forb$(\mathcal{F})$ denote the class of finite $L^-$-structures  $\pmb{A}$
such that no member of $\mathcal{F}$ embeds into $\pmb{A}$.
It is well-known that a  \Fraisse\ class  in signature  $L^-$ has free amalgamation if and only if it is of the form Forb$(\mathcal{F})$ for  some set $\mathcal{F}$ of finite irreducible $L^-$-structures.
It follows from results of
 Ne\v{s}et\v{r}il and R\"{o}dl in \cite{JV77,JV83} that all \Fraisse\ classes in signature $L$ for which the $L^-$-reduct has free amalgamation
has the Ramsey property.

For $k\geq 3$, a graph $G$ is called  {\em $k$-clique free} if for any $k$ vertices in $G$, there is at least
one pair with no edge between them; in other words, no $k$-clique embeds into $G$ as
an induced subgraph.
 An oriented graph $G=\langle V(G), E(G)\rangle$ is  a relational structure, where $V(G)$
denotes its vertex set and $E(G)\subseteq V(G)\times V(G)$ denotes its  directed edge relation, that is,  $E(G)\subseteq V(G)\times V(G)$ is an irreflexive binary relation  such that for all $x, y\in V(G)$,
$(x, y)\in E(G)$ implies  $ (y, x)\notin E(G)$. A tournament $G$ is an
oriented graph  such that for all $x\neq y$,  either $(x, y)\in E(G)$ or $(y, x)\in E(G)$.
A partial order with a linear extension is a structure  $\pmb{P}=\langle P,<^{\pmb{P}},R^{\pmb{P}}\rangle$ where $R^{\pmb{P}}$ is a partial  ordering on  $P$,
$<^{\pmb{P}}$ is a linear ordering on $P$,  and whenever $x\ne y$ and $R^{\pmb{P}}(x,y)$ holds, then also $x<^{\pmb{P}}y$ holds.

\begin{example}
Let $\mathcal{OG}$, $\mathcal{OG}_k$,\ $\mathcal{OOG}$,  $\mathcal{OT}$, and $\mathcal{OPO}$ denote  the  Fra\"{\i}ss\'{e} classes of all finite  ordered graphs,  finite  ordered $k$-clique free graphs  ($k\geq 3$), finite  ordered oriented graphs,  finite  ordered tournaments, and finite partial orders with a linear extension, respectively.
Each of these classes
has the Ramsey property.

The Ramsey property for
 $\mathcal{OG}$,  $\mathcal{OG}_k$, $\mathcal{OOG}$, and  $\mathcal{OT}$,
   are special cases of a  theorem of
   Ne\v{s}et\v{r}il-R\"{o}dl (\cite{JV77,JV83});
   the Ramsey property for
   that $\mathcal{OG}$ and $\mathcal{OT}$ follow from independent work of Abramson and Harrington in \cite{AH78}.
 The Ramsey property for  $\mathcal{OPO}$ was announced by  Ne\v{s}et\v{r}il and R\"{o}dl in \cite{NR84}, and the first proof was published
 by Paoli, Trotter, and Walker
  in \cite{PTW85}.
\end{example}

Let $\mathcal{K}$ be a Fra\"{\i}ss\'{e} class of  finite ordered 	binary relational structures  with the Ramsey property.
We may assume the universe of Flim$(\mathcal{K})$
 is  $\omega$,   so that the universe is well-ordered.
For $n\in \omega$, let Flim$(\mathcal{K})_n$ be the initial segment of   Flim$(\mathcal{K})$ on $\{0,\dots,n\}$.

\begin{definition}\label{defn.T_max}
Let $\mathcal{K}$ be a Fra\"{\i}ss\'{e} class of finite ordered binary relational structures with the Ramsey property, and let $T_{\max}\subseteq \omega^{<\omega}$ be the nonempty tree such that for each $t\in  T_{\max}$,
$$
\mbox{succ}_{T_{\max}}(t)=\{t^\frown\langle0\rangle, t^\frown\langle1\rangle, \dots,  t^\frown\langle|t|\rangle\}.
$$
For $t\in T_{\max}$, we define $\pmb{F_t}\in \mathcal{K}$  to have universe
 $F_t:=\mbox{succ}_{T_{\max}}(t)$, ordered by the
lexicographical ordering, such that $\pmb{F_t}$ is isomorphic to  Flim$(\mathcal{K})_{|t|}$.
\end{definition}

\begin{definition}\label{defn.F_max}
Let  $\mathcal{K}$ be a     Fra\"{\i}ss\'{e} class  of   finite ordered binary relational structures  with the Ramsey property.
The tree $T_{\max}$ and $\pmb{F_t}\in \mathcal{K}  \ (t\in T_{\max})$ induce a structure $\pmb{F_{\max}}$  on the universe $F_{\max}:=[T_{\max}]$, ordered lexicographically, with the binary relations $R^{\pmb{F_{\max}}}_\ell$, $\ell<k$
(where $k$ is the cardinality of the signature $L$),
 as follows:
$$
\forall \ x, y\in [T_{\max}],\  (x, y)\in
R_\ell^{\pmb{F_{\max}}}\Longleftrightarrow
(x\upharpoonright (|x\cap y|+1), y\upharpoonright (|x\cap y|+1))\in R_\ell^{\pmb{F_{x\cap y}}}.
$$
\end{definition}

\begin{lemma}
Let $\mathcal{K}$ be a Fra\"{\i}ss\'{e} class of finite ordered binary relational structures with the Ramsey property.
If $\pmb{F}\in \mathcal{K}$  and $t\in T_{\max}$, then there is $s\in T_{\max}$
with $t\subseteq s$ such that $\pmb{F}$ embeds into $\pmb{F_s}$.
\end{lemma}

\begin{proof}
Since  $\pmb{F_t}$ is isomorphic to  Flim$(\mathcal{K})_{|t|}$,
 it suffices
to prove that each $\pmb{F}\in \mathcal{K}$ embeds into   Flim$(\mathcal{K})$ on universe $\omega$.
Now the age of Flim$(\mathcal{K})$ is exactly $\mathcal{K}$. So each $\pmb{F}\in \mathcal{K}$ embeds into   Flim$(\mathcal{K})$ on  universe $\omega$.
\end{proof}

\begin{definition}\label{defn.2.6}
(1) Suppose that $T$ is  a  subtree of $T_{\max}$ and $t\in T$. Let
$\pmb{F^T_t}$ denote the induced substructure of $\pmb{F_t}$ on the universe $F^T_t:=\mbox{succ}_{T}(t)$. We define $\pmb{F(T)}$ to
be the induced  substructure of $\pmb{F_{\max}}$ on the universe $F(T):=[T]$.

(2) Let $\mathcal{K}$ be a Fra\"{\i}ss\'{e} class of finite ordered binary relational structures    with the Ramsey property. A subtree $T$ of $T_{\max}$ is called an
{\em $\pmb{F_{\max}}$-tree}
 if for every  $\pmb{F}\in \mathcal{K}$  and every $t\in T$, there is $s\in T$
with $t\subseteq s$ such that $\pmb{F}$ embeds into $\pmb{F^T_s}$.
In particular, $T_{\max}$ is an $\pmb{F_{\max}}$-tree.
\end{definition}

\begin{definition}
Let $\mathcal{K}$ be a Fra\"{\i}ss\'{e} class of finite ordered binary relational structures with the Ramsey property. A sequence $(T_j)_{j\in \omega}$ is a {\em fusion sequence} with witness $(m_j)_{j\in \omega}$ if the
following hold:
\begin{enumerate}
\item
 $(m_j)_{j\in \omega}$ is a strictly increasing sequence of natural numbers.
\item
 For all $j, l\in  \omega$, if $j< l$, then $T_l$
is an $\pmb{F_{\max}}$-subtree of $T_j$ such that
$T_j(m_j)=T_l(m_j)$.
\item
 For every $\pmb{F}\in \mathcal{K}$ , every $j\in \omega$, and every $t\in T_j(m_j)$,
there is $l> j$ such that $t$ has an extension $s$ in $T_l$ such that $|s|< m_l$ and
$\pmb{F}$ embeds into $\pmb{F^{T_l}_s}$.
\end{enumerate}
\end{definition}

One can   check that if $(T_j)_{j\in \omega}$ is a fusion sequence witnessed by
$(m_j)_{j\in \omega}$, then the fusion $\bigcap_{j\in \omega}
T_j=
\bigcup_{j\in \omega}
(T_j\cap \omega^{\leq m_j})$ is an $\pmb{F_{\max}}$-tree.


\section{Types}


\begin{definition}
Let $T$  be an $\pmb{F_{\max}}$-tree  and $\pmb{F}$ a finite induced substructure of $\pmb{F(T)}$.
We define $\triangle(\pmb{F})$ and $\pmb{F}^{\vee}$ as follows:
$$\triangle(\pmb{F})=\max\{|x\cap y| : x, y\in F\wedge x\neq y\},$$
$$\pmb{F}^{\vee}=\{x\upharpoonright (\triangle(\pmb{F})+1) : x\in F\}.$$
\end{definition}

\begin{example}
Let $\mathcal{K}=\mathcal{OG}$ and $\pmb{H}\in \mathcal{K}$ as in Figure 1, where $H=\{x, y, z\}$,  $x=0000\dots$,
$y=0100\dots$, and $z=0111\dots$.  Then $\triangle(\pmb{H})=2$ and $\pmb{H}^{\vee}=\{u, v, w\}$, where $u=000, v=010$ and $w=011$.

\begin{center}
\centering
\includegraphics[totalheight=2.2in]{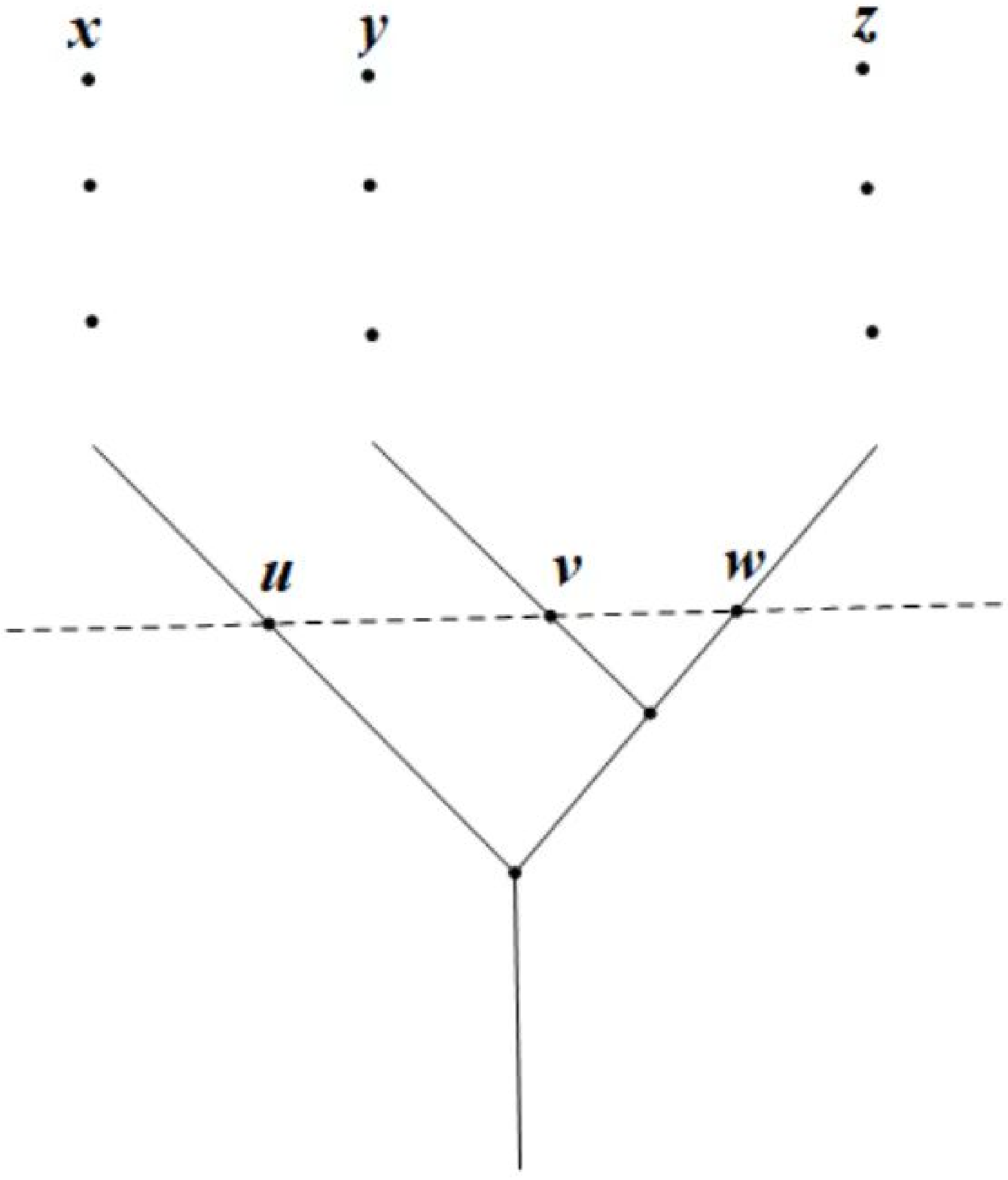}\\Figure 1:\ $\triangle(\pmb{H})$ and $\pmb{H}^{\vee}$
\end{center}
\end{example}

\begin{definition}
Let $\pmb{F}$ and $\pmb{F^{\prime}}$ be finite induced substructure of $\pmb{F_{\max}}$. We
say  $\pmb{F}$ and $\pmb{F^{\prime}}$    are {\em strongly isomorphic} if there exists an isomorphism $\varphi :  \pmb{F}\longrightarrow \pmb{F^{\prime}}$
such that $\forall \ \{x_0, y_0\}, \{x_1, y_1\}\in [F]^2$,
$$|x_0\cap y_0|\leq |x_1\cap y_1|\Longleftrightarrow |\varphi(x_0)\cap \varphi(y_0)|\leq |\varphi(x_1)\cap \varphi(y_1)|.$$
\end{definition}

Clearly, strong isomorphism is an equivalence relation. By a {\em type} we mean a strong
isomorphism equivalence class. In particular,
there are only finitely many types inside an isomorphism class.

Suppose that $\pmb{F}$ and $\pmb{H}$ are structures. Let $\big(\mathop{}_{\pmb{H}}^{\pmb{F}}\big)$
be the set of all
induced substructures $\pmb{H^{\prime}}$ of $\pmb{F}$ isomorphic to $\pmb{H}$. If  $\pmb{F}$ is a  induced  substructure of $\pmb{F_{\max}}$  and $\tau$ is a type,
we let $\big(\mathop{}_{\tau}^{\pmb{F}}\big)$
be the set of all induced  substructures of $\pmb{F}$ of type  $\tau$.

\begin{example}

Let $\mathcal{K}=\mathcal{OG}$ and $\pmb{H}\in \mathcal{K}$ be as in Figure 2, where $H=\{v_0, v_1, v_2\}$.

\begin{center}
\centering
\includegraphics[totalheight=0.4in]{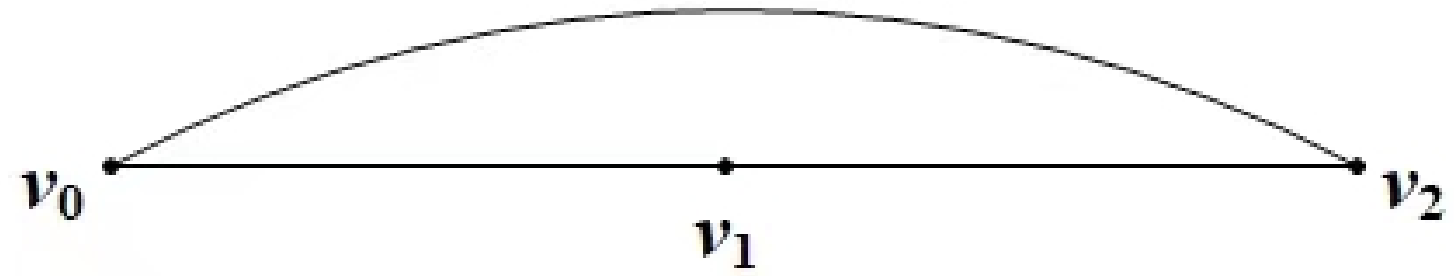}\\Figure 2
\end{center}
Then there are   3 types for $\pmb{H}$ as in Figure 3.
\begin{center}
\centering
\includegraphics[totalheight=2.0in]{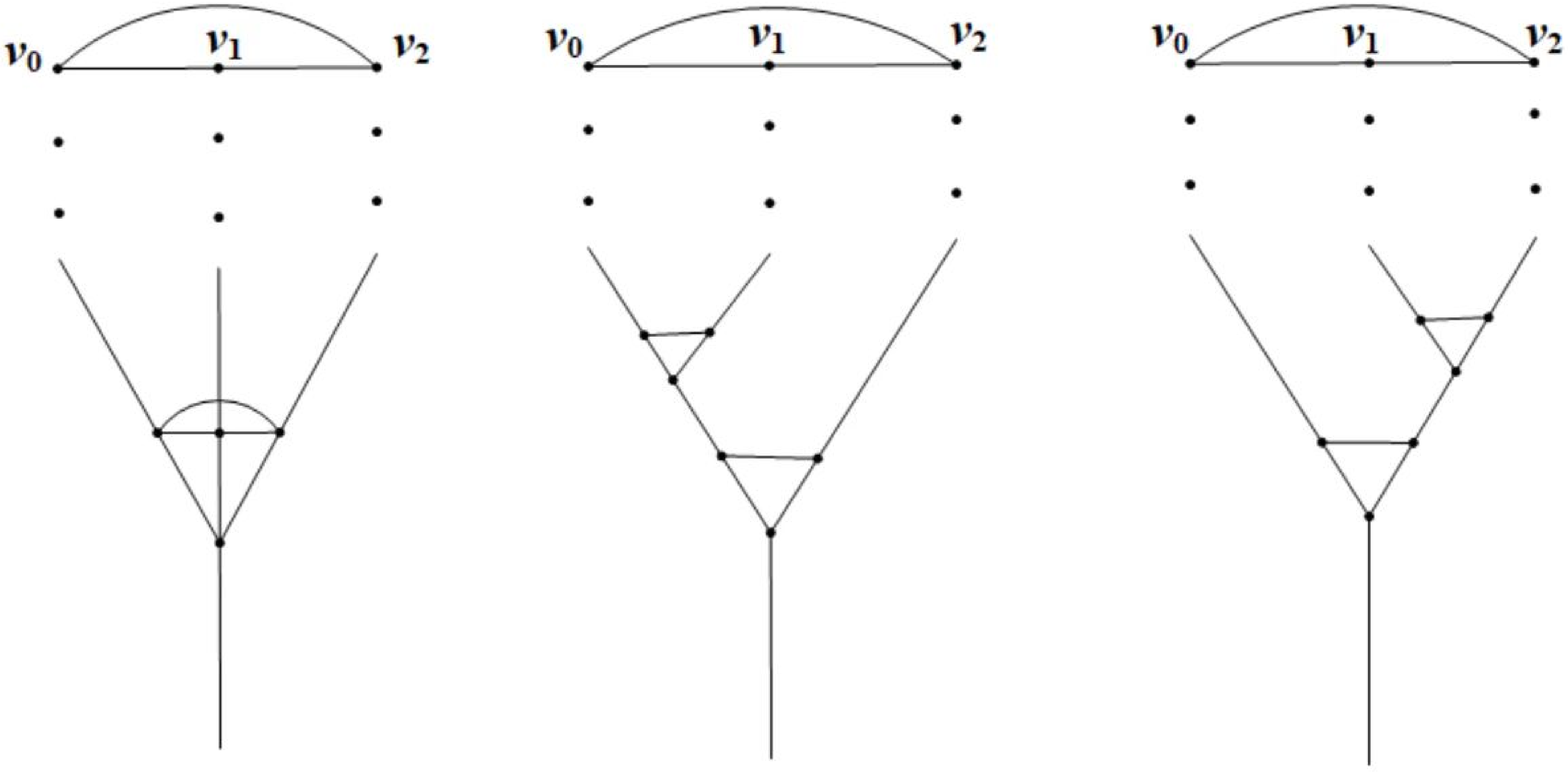}\\Figure 3:\ 3 types for $\pmb{H}$
\end{center}

\noindent Let $\pmb{H}\in \mathcal{K}$ be as in Figure 4. Then there are   2 types for $\pmb{H}$  as in Figure 5.
\begin{center}
\centering
\includegraphics[totalheight=0.4in]{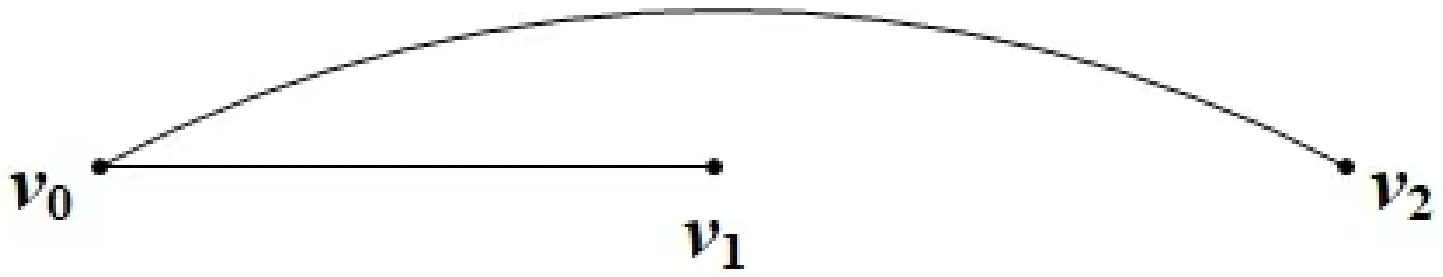}\\Figure 4
\end{center}

\begin{center}
\centering
\includegraphics[totalheight=2.2in]{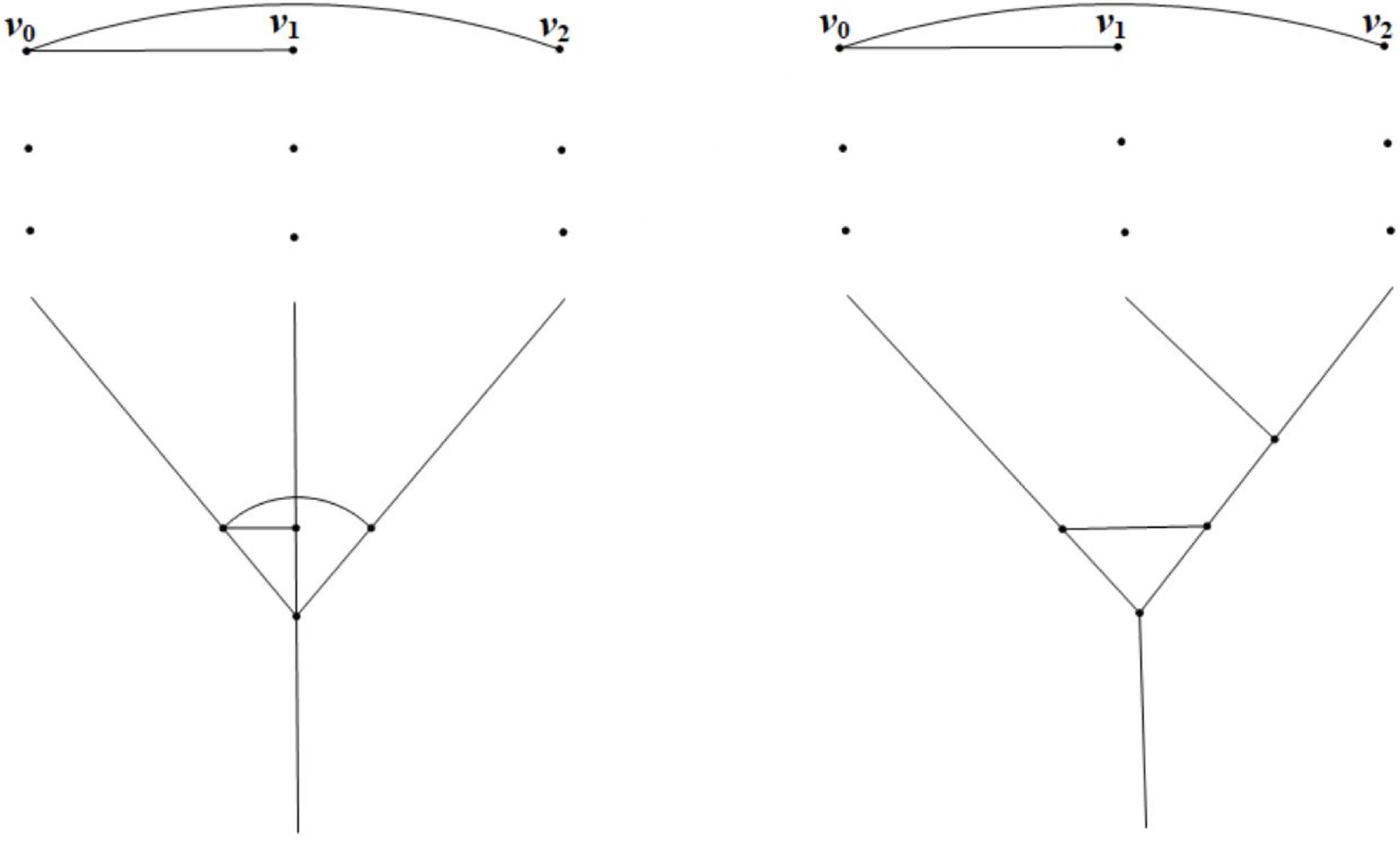}\\Figure 5:\ 2 types  for $\pmb{H}$
\end{center}

\vspace*{0.2cm}

\noindent Let $\pmb{H}\in \mathcal{K}$ be as in Figure 6. Then there is only 1 type  for $\pmb{H}$ as in Figure 7.

\begin{center}
\centering
\includegraphics[totalheight=0.4in]{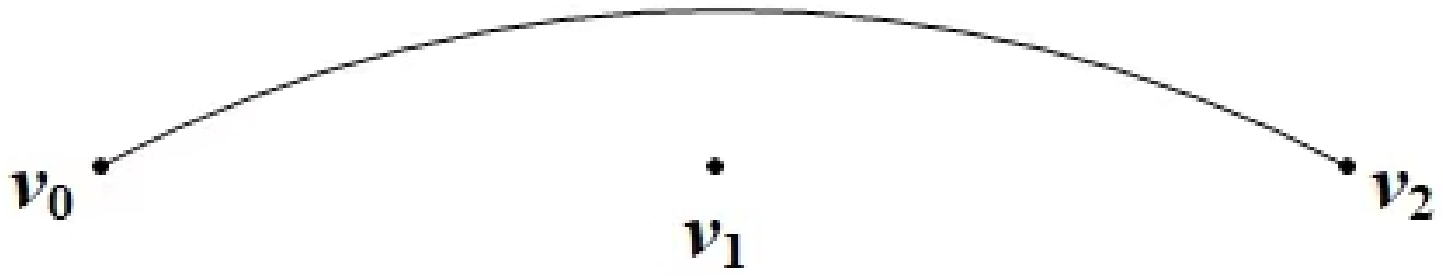}\\Figure 6
\end{center}

\begin{center}
\centering
\includegraphics[totalheight=2.2in]{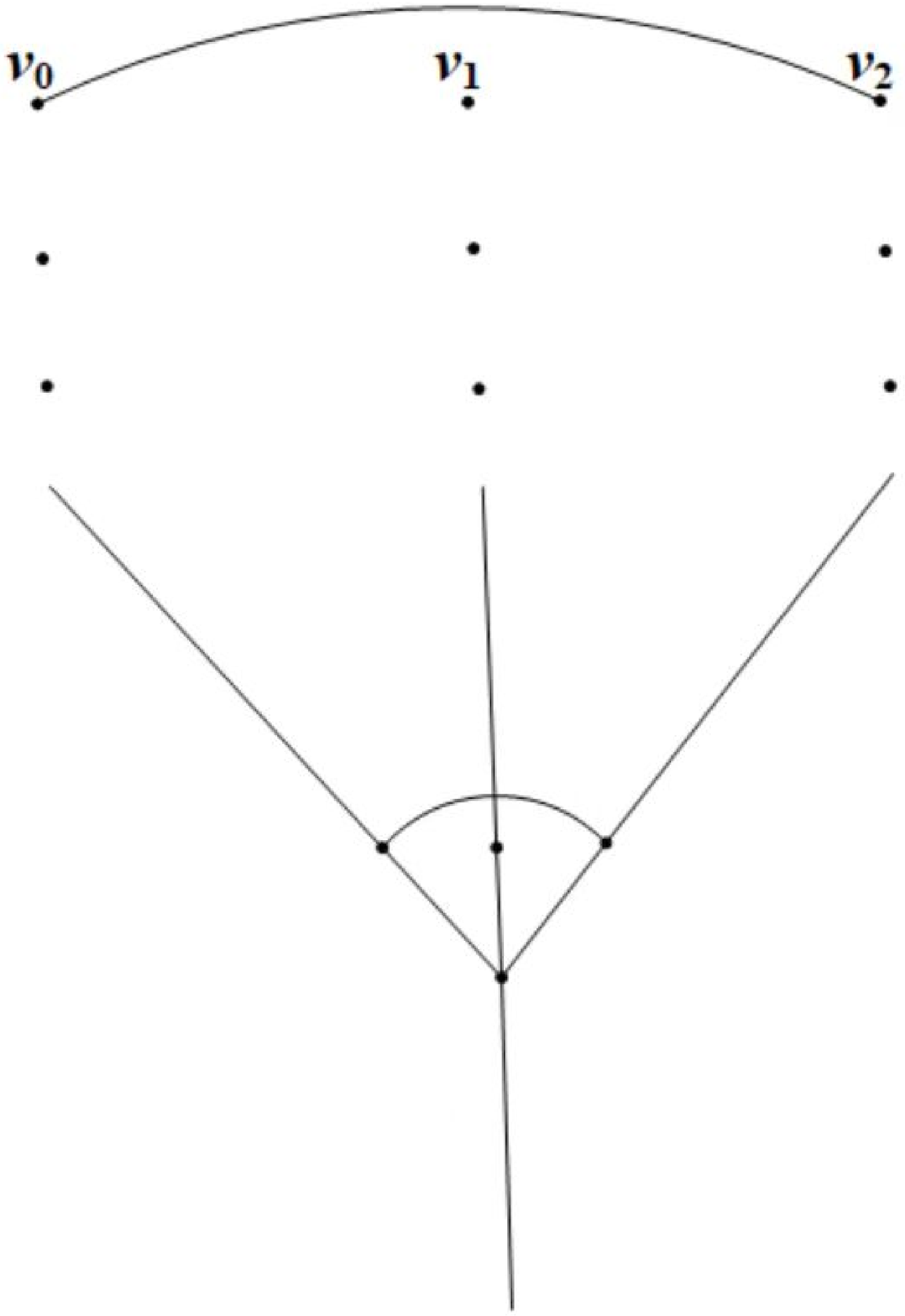}\\Figure 7:\ 1 type  for $\pmb{H}$
\end{center}

\end{example}

\begin{example}

Let $\mathcal{K}=\mathcal{OT}$ and $\pmb{F}\in \mathcal{K}$ be  as in Figure 8.

\begin{center}
\centering
\includegraphics[totalheight=0.4in]{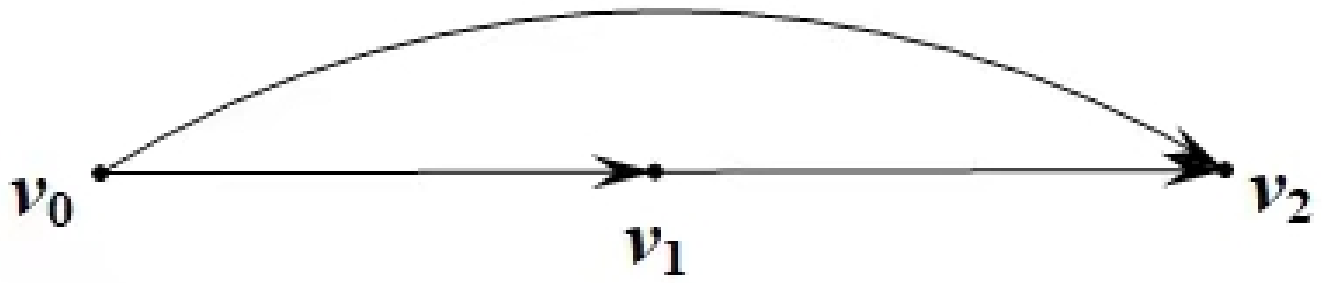}\\Figure 8
\end{center}

\noindent Then there are   3 types for $\pmb{F}$  as in Figure 9.

\begin{center}
\centering
\includegraphics[totalheight=2.2in]{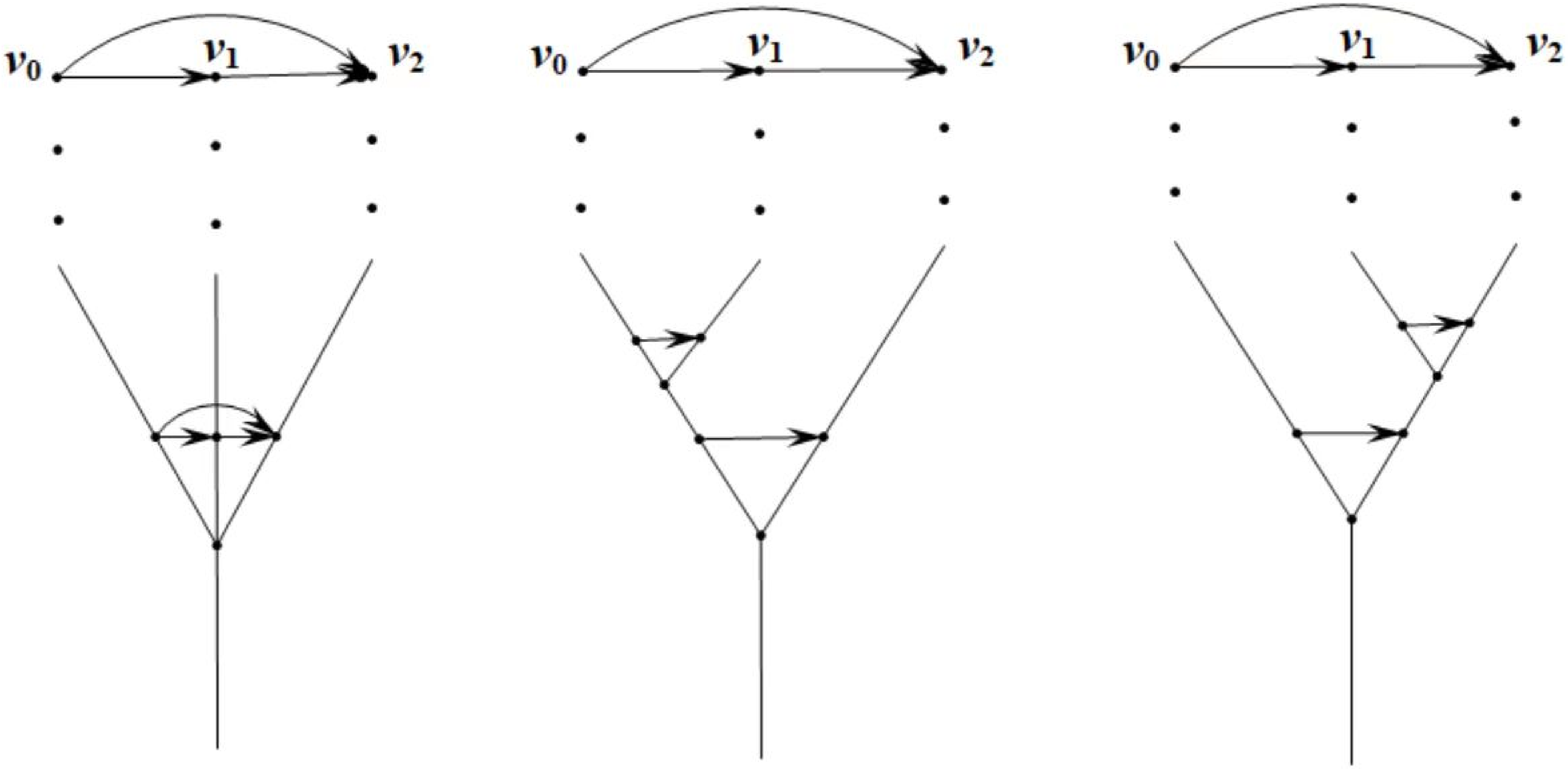}\\Figure 9:\ 3 types  for $\pmb{F}$
\end{center}

\noindent If  $\pmb{F}\in \mathcal{K}$ is as in Figure 10, then there are 2 types  for $\pmb{F}$ as in Figure 11.

\begin{center}
\centering
\includegraphics[totalheight=0.4in]{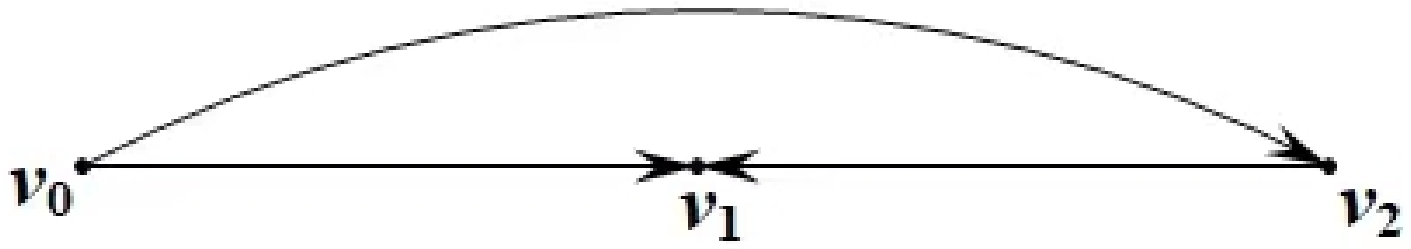}\\Figure 10
\end{center}

\begin{center}
\centering
\includegraphics[totalheight=2.2in]{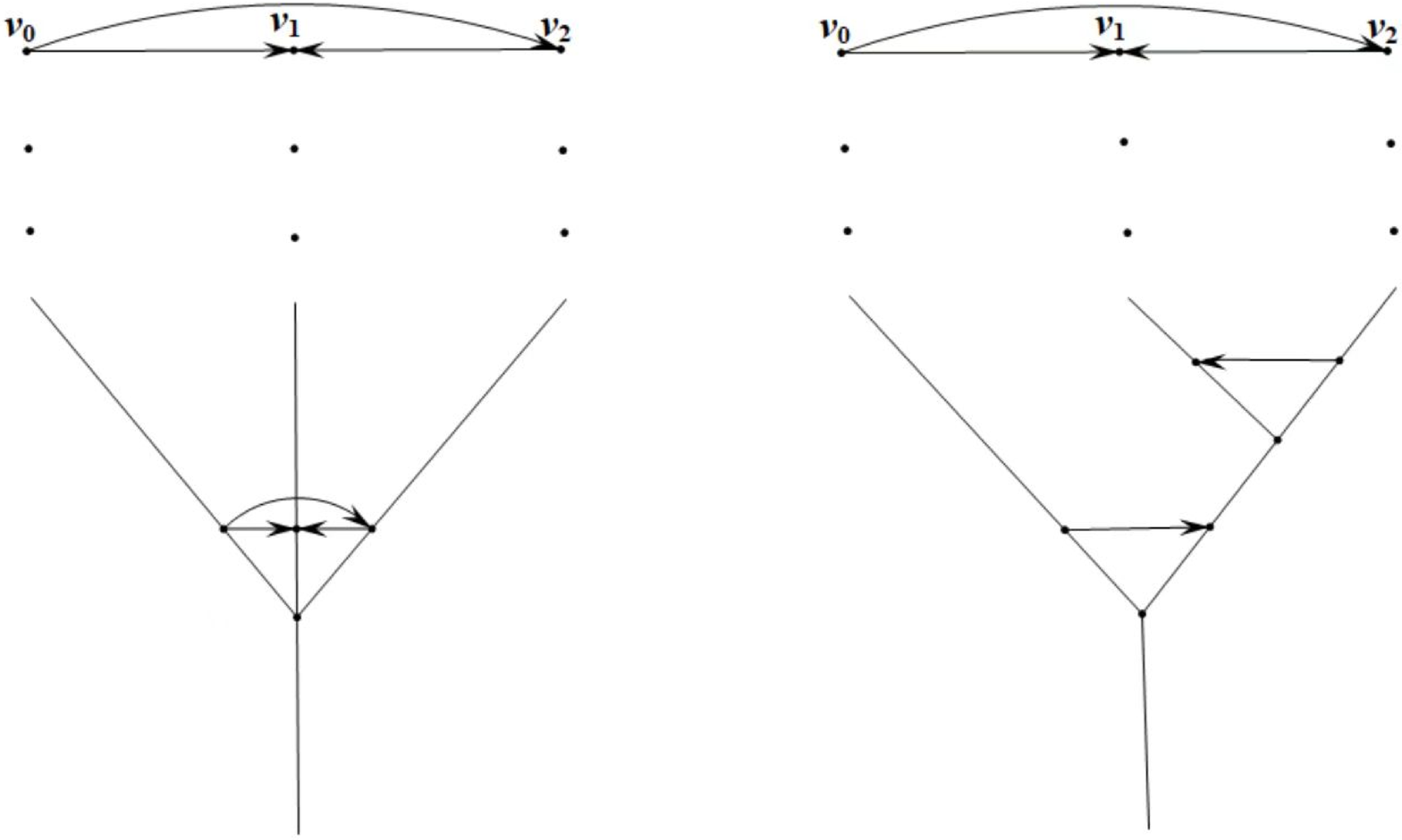}\\Figure 11:\ 2 types  for $\pmb{F}$
\end{center}

Let  $\pmb{F}\in \mathcal{K}$ be as in Figure 12.

\begin{center}
\centering
\includegraphics[totalheight=0.4in]{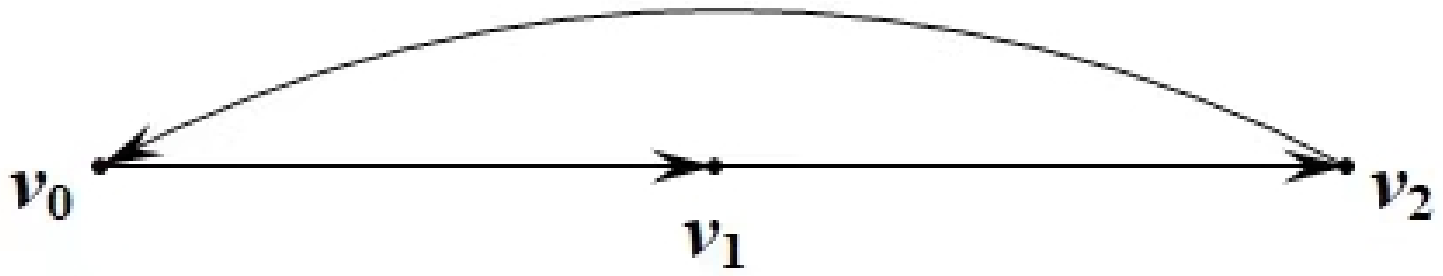}\\Figure 12
\end{center}

Then there is  only 1 type  for $\pmb{F}$ as in Figure 13.

\begin{center}
\centering
\includegraphics[totalheight=2.2in]{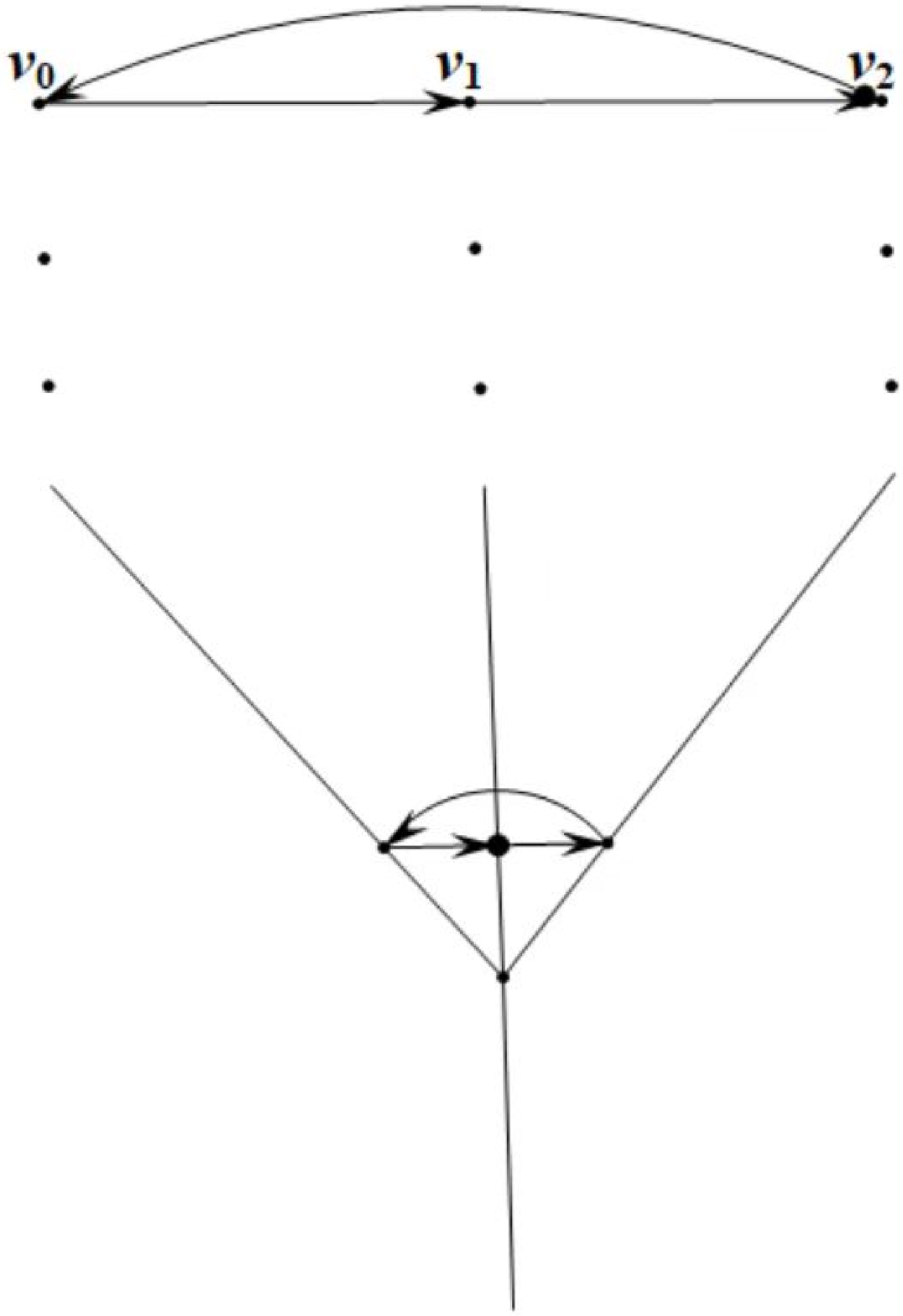}\\Figure 13:\ 1 type  for $\pmb{F}$
\end{center}

\end{example}

\begin{example}

Suppose that  $\mathcal{K}=\mathcal{OPO}$ and $\pmb{H}\in \mathcal{K}$ as in Figure 14.

\begin{center}
\centering
\includegraphics[totalheight=0.5in]{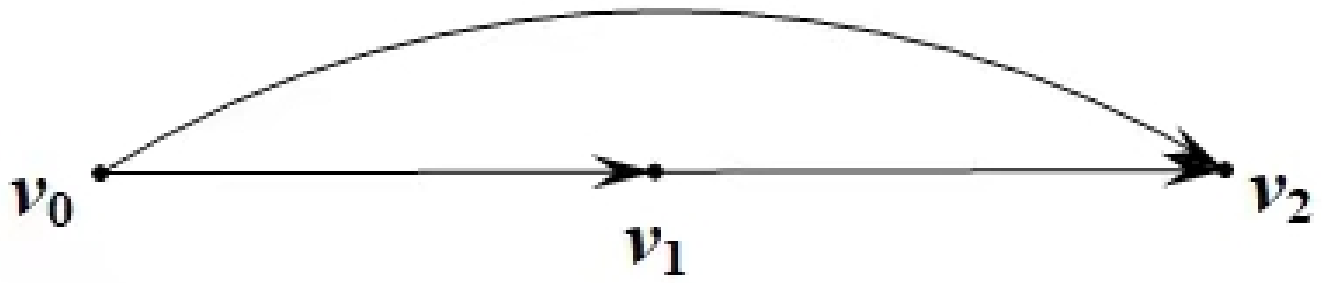}\\Figure 14
\end{center}

\noindent Here

\begin{center}
\centering
\includegraphics[totalheight=0.25in]{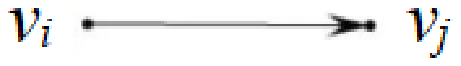}
\end{center}
denotes  that $R(v_i, v_j)$, where $R$ is a partial order. Then there are 3 types  for $\pmb{H}$ as in Figure 15.

\begin{center}
\centering
\includegraphics[totalheight=2.2in]{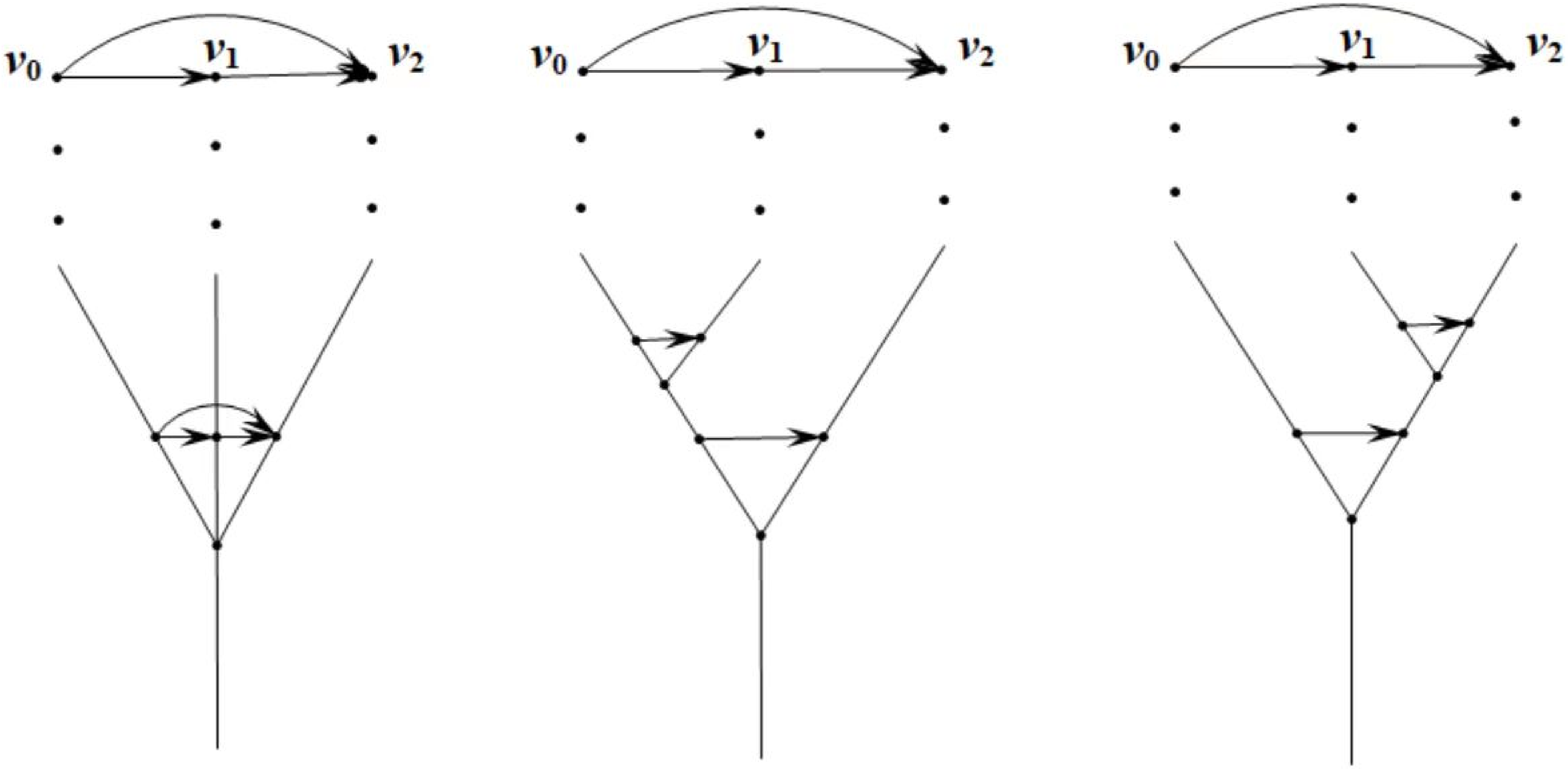}\\Figure 15:\ 3 types  for $\pmb{H}$
\end{center}

\vspace*{0.3cm}

\noindent Let $\pmb{H}\in \mathcal{K}$ be as in Figure 16. Then there are 2 types  for $\pmb{H}$ as in Figure 17.

\begin{center}
\centering
\includegraphics[totalheight=0.5in]{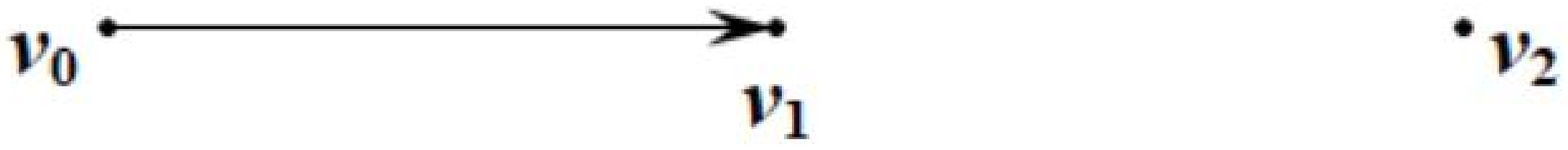}\\Figure 16
\end{center}

\begin{center}
\centering
\includegraphics[totalheight=2.2in]{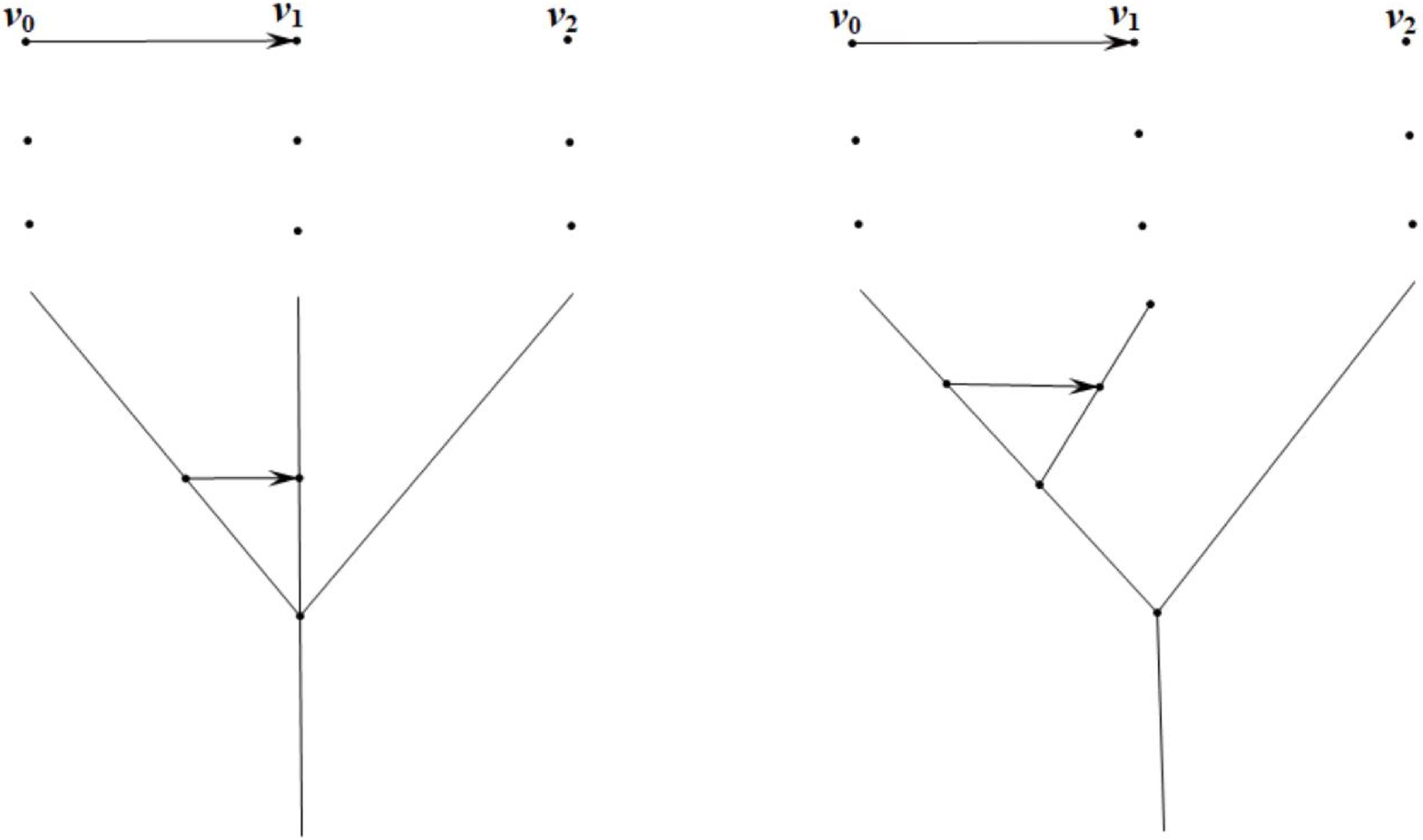}\\Figure 17:\ 2 types  for $\pmb{H}$
\end{center}

\vspace*{0.3cm}

\noindent Let $\pmb{H}\in \mathcal{K}$ be as in Figure 18. Then there is only 1 type  for $\pmb{H}$ as in Figure 19.

\vspace*{0.3cm}

\begin{center}
\centering
\includegraphics[totalheight=0.4in]{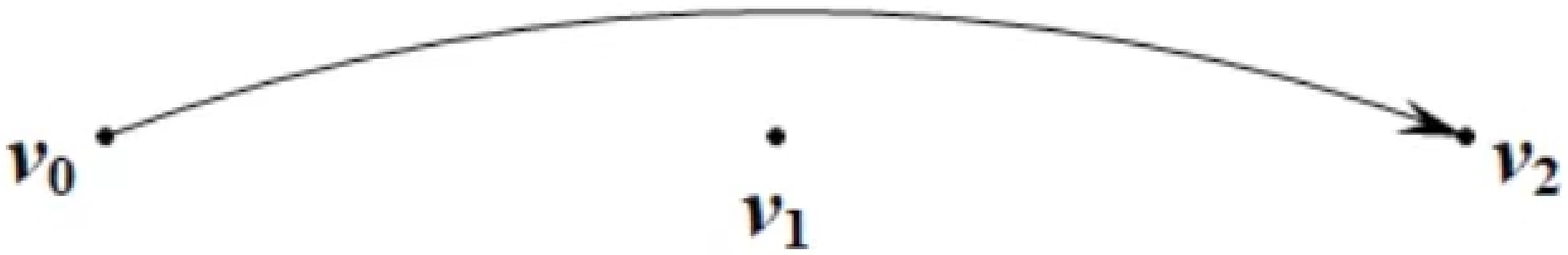}\\Figure 18
\end{center}

\vspace*{0.2cm}

\begin{center}
\centering
\includegraphics[totalheight=2.2in]{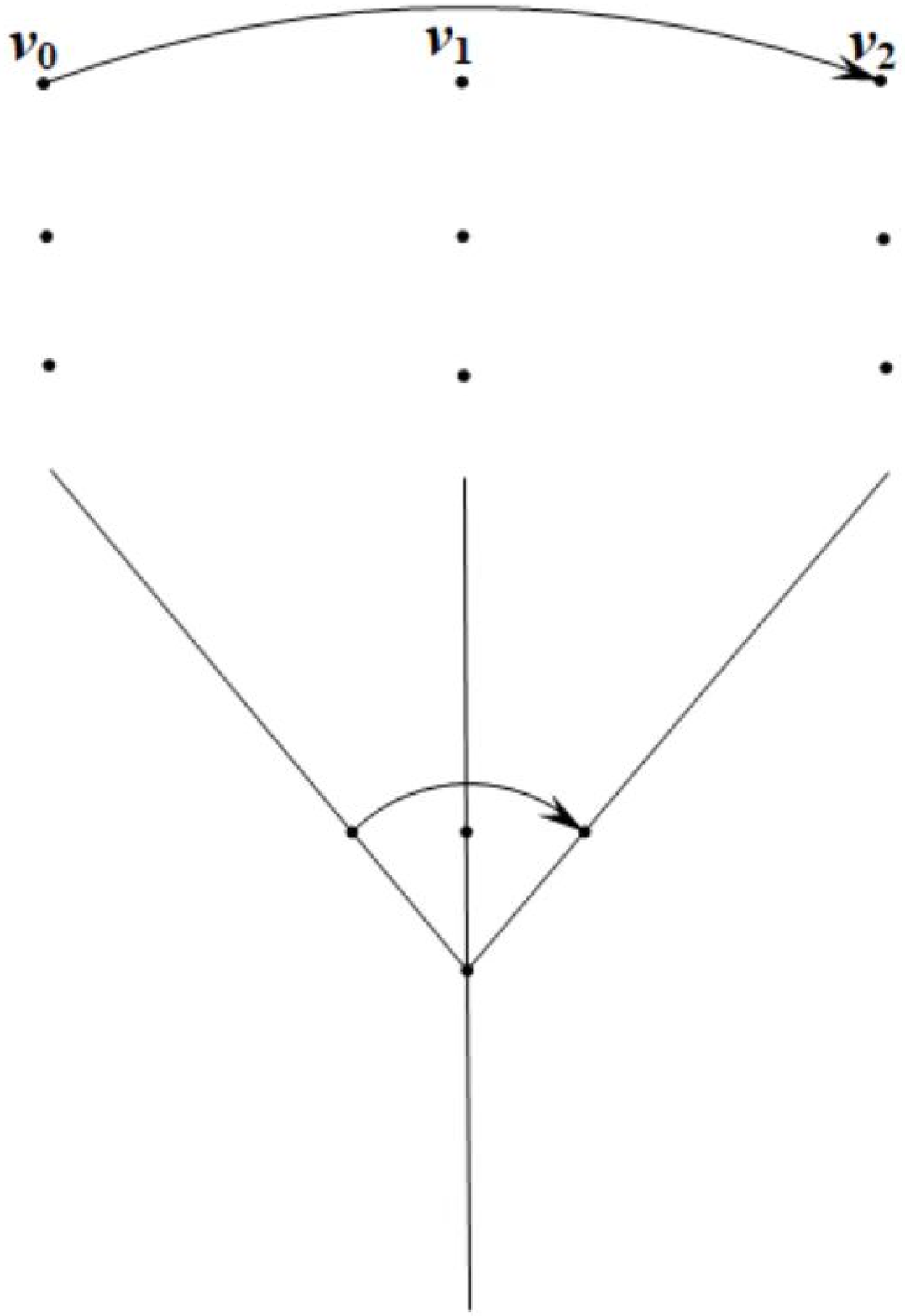}\\Figure 19:\ 1 type  for $\pmb{H}$
\end{center}

\end{example}

\vspace*{0.1cm}



\section{Topological Ramsey spaces for coding  inverse limit structures for finitely many  binary relations}


This section is essentially  work of Zheng from Section 3 in \cite{YY18}.
Her work is straightforwardly extended  from the  context of finite  ordered graphs to
the broader context of Fra\"{\i}ss\'{e}  classes of  finite ordered binary relational structures satisfying the Ramsey property.
We include it in this paper for the reader's convenience, making a few modifications.

Let $\mathcal{K}$ be a Fra\"{\i}ss\'{e} class of finite ordered binary relational structures satisfying the Ramsey property, with  signature $L=\{<, R_0,\dots,R_{k-1}\}$, where each $R_{\ell}$  $(\ell<k)$ is a binary relation.
We fix a type $\tau$ and build a topological Ramsey space $\mathcal{F}_\infty(\tau)$.
We may denote this space by $\mathcal{F}_\infty$ when the type is clear
from the context.
Let $m+1$ be the number of elements for a finite ordered structure in  $\tau$.
Let $\{\pmb{F_i} : i< \omega\}$ enumerate the set of all finite structures in $\mathcal{K}$, up to isomorphism, labelled
so that  for every $i < j < \omega$,  $|F_i| \le |F_j|$.

Let $S$ be a tree.
A node $t\in S$ is a {\em splitting node} if $|\mbox{succ}_S (t)| > 1$. We
say $S$ is {\em skew} if $S$ has at most one splitting node at each level, i.e.,
$$
\forall \ n\in  \omega,\ |\{t\in  S\cap \omega^n : |\mbox{succ}_S(t)| > 1\}|\le 1.
$$
Notice  that if
 $S$ is  a skew tree and  $i>0$ is given,
then each node
 $t\in S$  for which $\pmb{F^S_t}
\cong \pmb{F_i}$  is a splitting node, since the structure $\pmb{F_i}$  has universe of size at least two.
Thus,  any two nodes in the set
$\{t\in  S : \pmb{F^S_t}
\cong \pmb{F_i}\}$ have different lengths, so the nodes in this set can be enumerated in order of increasing length.  This will be useful in part (iii) of (2) in the next definition.

\begin{definition}  Let $\tau$ be a type and $m+1$ be the number of elements for a finite ordered structure in  $\tau$. We define the space $(\mathcal{F}_\infty, \leq, r)$ as follows.

Let $S$ be a member of   $\mathcal{F}_\infty$ if $S$ is a skew subtree of $T_{\max}$ and when we enumerate the set of
splitting nodes $\{s\in S : |\mbox{succ}_S(s)|> 1\}$ as $\{s_i\}_{i<\omega}$ in the order of length,

(1) there is a finite structure $\pmb{F}\in \tau$  such that
$$\pmb{F}^{\vee}= S\cap \omega^{|s_{m-1}|+1};$$

(2) for all $i> 0$,

\quad (i) $\forall \ s\in  S\cap \omega^{>|s_{m-1}|}, \forall \ u\in \mbox{succ}_S(s)$,
$$\pmb{F^S_{s}}\cong \pmb{F_i}\Rightarrow \exists!\ t\in S\
(u\subseteq t\wedge \pmb{F^S_t}\cong \pmb{F_{i+1}});$$

\quad (ii)  for every pair $s,\ t\in  S\cap \omega^{>|s_{m-1}|}$,
$$\pmb{F^S_s}\cong \pmb{F_i}\wedge \pmb{F^S_t}\cong \pmb{F_{i+1}}\Rightarrow |s| < |t|;$$

\quad (iii) if $\{t\in  S : \pmb{F^S_t}
\cong \pmb{F_i}\}$ is enumerated in  order of increasing length as $\{t_j\}_j$, then
there is $l< \omega$ such that $\{t_j\upharpoonright l\}_j$ is strictly increasing in lexicographical
ordering.
\end{definition}

When we say that $\{s_i\}_{i<\omega}$ is the set of splitting nodes in $S$, we tacitly assume
that the length $|s_i|$ is strictly increasing in $i$.

For $S, U\in \mathcal{F}_\infty$, we write  $S\leq U$ if  and only if $S\subseteq U$. For $l< \omega$ and $S\in \mathcal{F}_\infty$ with the set of splitting nodes $\{s_i\}_{i<\omega}$, we define the finite
approximation $r_l(S)$ as follows: let $$r_0(S)= \emptyset \ \mbox{and} \  r_{l+1}(S) = S\cap \omega^{\leq |s_l|+1}.$$
We specify a few
more definitions that are often used in topological Ramsey spaces. Let $\mathcal{F}_{< \infty}$ denote
the set of all finite approximations, i.e.
$$\mathcal{F}_{< \infty}=\{r_l(S) : S\in \mathcal{F}_{\infty}\wedge l\in \omega\}.$$
For $a, b\in \mathcal{F}_{< \infty}$, let $a \leq_{\small\mbox{fin}} b$   if $a\subseteq b$. Let $|a|=n$ if there is $S\in \mathcal{F}_{\infty}$ with $r_n(S)=a$.
For $a, b\in  \mathcal{F}_{< \infty}$, we write $a\sqsubseteq b$ if there are $l< p<\omega$ and $S\in \mathcal{F}_{\infty}$ such that
$a=r_l(S)$ and $b=r_p(S)$.

For $a\in  \mathcal{F}_{< \infty}$ and $S\in \mathcal{F}_{\infty}$,
depth$_S(a)=\min\{n : a\leq_{\small\mbox{fin}}r_n(S)\}$,
where by convention, $\min \emptyset=\infty$.
We equip the space $\mathcal{F}_{\infty}$ with the Ellentuck topology,
with basic open sets of the form
$$[a, S]=\{X\in  \mathcal{F}_{\infty} : (X\leq S)\wedge (\exists l)(r_l(X)=a)\},$$ for $a\in  \mathcal{F}_{< \infty}$ and $S\in \mathcal{F}_{\infty}$. For $l< \omega$, let
$$\mathcal{F}_l=\{r_l(X) : X\in \mathcal{F}_{\infty}\},$$
$$[l, S]=[r_l(S), S], \mbox{and}$$
$$r_l[a, S]=\{r_l(X) : X\in [a, S]\}.$$
The height of an element $a\in \mathcal{F}_{< \infty}$ is height$(a)=\max_{s\in a} |s|$. In general, $|a|\leq\mbox{height}(a)$.

Now we show that $(\mathcal{F}_{\infty}, \leq, r)$ is a topological Ramsey space by proving that $\mathcal{F}_{\infty}$ is
closed
as a subspace of
$(\mathcal{F}_{< \infty})^\omega$,
and  satisfies the axioms {\bf (A1)-(A4)} as defined in pages 93--94 of  \cite{ST10}.
It is straightforward to check {\bf (A1)-(A3)}.
Moreover,  $\mathcal{F}_{\infty}$ is a closed subset of $(\mathcal{F}_{< \infty})^\omega$ when we identify $S\in  \mathcal{F}_{\infty}$
with $(r_n(S))_{n< \omega}\in (\mathcal{F}_{< \infty})^\omega$  and equip $\mathcal{F}_{< \infty}$ with the discrete topology and $(\mathcal{F}_{< \infty})^\omega$ with the product topology.

\begin{definition}
Let $T\subseteq \omega^{<\omega}$ be a (downwards closed) finitely branching  tree  with no terminal nodes, and let
 $N$ an infinite subset of $\omega$.
A set $U$ is called  a {\em strong subtree}
 of $\bigcup_{n\in N} T\cap \omega^n$ if there is an infinite set $M\subseteq N$ such
that the following conditions hold.
\begin{enumerate}
\item
$U\subseteq \bigcup_{n\in M} T\cap \omega^m$  and $U\cap \omega^m\neq \emptyset$ for all $m\in M$. In this case, we say that $M$
witnesses that $U$ is a strong subtree.
\item
If $m_1< m_2$ are two successive elements of $M$ and if $u\in U\cap \omega^{m_1}$, then every
immediate successor of $u$ in $\bigcup_{n\in N} T\cap \omega^n$  has exactly one extension
in $U\cap \omega^{m_2}$.
\end{enumerate}
\end{definition}

\begin{theorem}[Halpern-L\"{a}uchli \cite{JH66}]\label{thm.HL}
For each $i< d$,
 let $T_i\subseteq \omega^{<\omega}$ be a  finitely branching tree  with no terminal nodes,   let $N\in [\omega]^\omega$, where $d$ is any positive integer, and let
$c :
\bigcup_{n\in N}
\prod_{i<d}T_i\cap \omega^n\longrightarrow p$
be a finite coloring, where $p$ is any positive integer. Then there is an infinite subset $M\subseteq N$ and infinite
strong subtrees $U_i\subseteq \bigcup_{n\in N} T_i\cap \omega^n$ witnessed by $M$ such that $c$ is monochromatic on
$\bigcup_{m\in M}
\prod_{i<d}U_i\cap \omega^m$.
\end{theorem}

We observe that
every $S\in \mathcal{F}_{\infty}$ is an $\pmb{F_{\max}}$-tree and
every $\pmb{F_{\max}}$-tree contains some
$S\in \mathcal{F}_{\infty}$ as a subtree.
This is because for all $n< \omega$, we can find infinitely many $m\in \omega$
such that $\pmb{F_n}$ embeds into $\pmb{F_m}$
(in fact,  this holds for all but finitely many $m\in \omega$.)
We define $(\pmb{A_n})_{n<\omega}$ to be a sequence of finite structures in $\mathcal{K}$
such that for each $n< \omega$, the finite structures $\pmb{F_0},\dots, \pmb{F_n}$ embed into $\pmb{A_n}$.

Now we will prove {\bf (A4)} in Lemma 4.4.
For $\mathcal{K}=\mathcal{OG}$,
Zheng proved in
\cite{YY18}
that $(\mathcal{F}_{\infty}, \leq, r)$ is a topological Ramsey space.
So the axiom {\bf (A4)} holds for finite ordered graphs.
 It should be pointed out that the proof of Lemma
 \ref{lem.4.4}
  follows from  Zheng's proof for finite ordered graphs. However, for the convenience of the reader, we also present  the proof here.

\begin{lemma}\label{lem.4.4}
 The axiom {\bf (A4)} holds for $(\mathcal{F}_{\infty}, \leq, r)$, i.e. for $a\in  \mathcal{F}_{< \infty}$ and $S\in \mathcal{F}_{\infty}$, if
\em{depth}$_S(a)< \infty$ and $\mathcal{O}\subseteq \mathcal{F}_{|a|+1}$, then
there is $U\in [$\mbox{depth}$_S(a), S]$ such that  $r_{|a|+1}[a, U]\subseteq \mathcal{O}$ or $r_{|a|+1}[a, U]\cap \mathcal{O}=\emptyset$.
\end{lemma}

\begin{proof}Let $m_0=\mbox{height}(a)$. Then $a$ has $|a|$ many splitting nodes, and each element in $\mathcal{F}_{|a|+1}$ has $|a|+1$ many splitting nodes.
 In particular, we
can find $u\in a\cap \omega^{m_0}$ and $j<\omega$ such that for every $b\in  \mathcal{F}_{|a|+1}$ with $a\sqsubseteq b$,  there is a unique splitting node
 $t\in b_u$, and $\pmb{F^b_t}
\cong \pmb{F_j}$.
Let $a\cap \omega^{m_0}$ be enumerated as $u= v_0, v_1,\dots, v_d$.

Step 1.
Let us construct a subtree $X\subseteq S$  along
with a strictly increasing
sequence $(m_l)_{l<\omega}$ starting with $m_0=$ height$(a)$ such that the following conditions
hold for every $l< \omega$ and every $s\in X\cap \omega^{m_l}$:
\begin{enumerate}
\item[(a)] $X_s$ has a unique splitting node $t$ of length in $[m_l, m_{l+1})$.
\item[(b)]
 For the $t$ from condition (a), $\pmb{F^X_t}\cong \pmb{A_{j+l}}$.
\end{enumerate}

Let $t\in  S_u$. The set of tuples of {\em nephews} of $t$ is defined to be
$$\{(t_1,\dots, t_d) : t_i\in X_{v_i}\cap \omega^{|t|+1} \
\mbox{for} \ 1\leq i\leq d\}.$$
Suppose $(t_i)_{1\leq i\leq d}$ is a tuple of nephews of $t$. Each finite structure $\pmb{F}\in
\big(\mathop{}_{\pmb{F_j}}^{\pmb{F^X_t}}\big)$
together with
$(t_i)_{1\leq i\leq d}$ determines an element $b_{\pmb{F},(t_i)}\in  \mathcal{F}_{|a|+1}$, where the set of $\subseteq$-maximal nodes
in $b_{\pmb{F},(t_i)}$ is $F\cup \{t_i : 1\leq i\leq d\}$. With this notion of $b_{\pmb{F}, (t_i)}$  defined, we can state
another requirement for $X$.

(c) For every $t\in X$ and every tuple $(t_i)_{1\leq i\leq d}$ of nephews of $t$, the set
$\Big\{b_{\pmb{F},(t_i)} : \pmb{F}\in  \big(\mathop{}_{\pmb{F_j}}^{\pmb{F^X_t}}\big)\Big\}$ is either included or disjoint from $\mathcal{O}$.

We recursively construct sets $X(m_l)\subseteq \omega^{m_l}$.
Then $X$ will be the downward closure of
$\bigcup_{l<\omega} X(m_l)$ and $X\cap \omega^{m_l}=X(m_l)$. Start with $X(m_0)=\{v_i : i\leq d\}$.
Assume we have constructed $X(m_l)$. The number of extensions in $X(m_{l+1})$ for
each $s\in X(m_l)$ is prescribed.
In particular, for each $t\supseteq u$ in $X(m_{l+1})$, the set of (tuples of)
nephews of $t$ will be finite and of the same size, independent of $t$.
Let this size be  $k< \omega$.
 Since $\mathcal{K}$ is a Fra\"{\i}ss\'{e} class of finite ordered binary relational structures with the Ramsey property, there is a finite
ordered structure $\pmb{H}\in \mathcal{K}$ such that
$\pmb{H}\longrightarrow (\pmb{A_{j+l}})^{\pmb{F_{j}}}_{2^{k}}$
Since $S$ is an $\pmb{F_{\max}}$-tree, for each $s\in X(m_l)$ extending $v_i \ (1\leq i\leq d)$, $s$ has an
extension $t(s)\in S$ such that there is
$\pmb{F(s)}\in \big(\mathop{}_{\pmb{A_{j+l}}}^{\pmb{F^S_{t(s)}}}\big)$.
On the other hand, for each
$s\in X(m_l)$ extending $u$, $s$ has an extension $t(s)\in S$ such that $\pmb{H}$ embeds into $\pmb{F^S
_{t(s)}}$.
For each tuple $(t_i)_{1\leq i\leq d}$ of nephews of $t(s)$, there is a natural coloring $c : \big(\mathop{}_{ \ \pmb{F_{j}}}^{\pmb{F^S_{t(s)}}}\big)\longrightarrow 2$
depending on whether $b_{\pmb{F}, (t_i)}$ is in $\mathcal{O}$.
Thus, there are at most $k$ many
2-colorings.
These colorings can be encoded in a single $2^n$-coloring of  $\big(\mathop{}_{ \ \pmb{F_{j}}}^{\pmb{F^S_{t(s)}}}\big)$.
Then there is $\pmb{H(s)}\in  \big(\mathop{}_{\pmb{A_{j+l}}}^{\pmb{F^S_{t(s)}}}\big)$ such that the set $\big(\mathop{}_{ \ \pmb{F_{j}}}^{\pmb{H(s)}}\big)$
is monochromatic.
 Let
$$
m_{l+1}=\max\{|t(s)| : s\in X(m_l)\}+1.
$$
Suppose that  $X(m_{l+1})\subseteq S\cap \omega^{m_{l+1}}$ has the property  that every node in $(\bigcup_{s\in D} F(s))$
$\cup (\bigcup_{s\in X(m_l)\backslash D} H(s))$ has a unique
extension in $X(m_{l+1})$, where $D=\{s\in X(m_l) : s$ extends $v_i,\ 1\leq i\leq d\}$. Thus $X$ satisfies (a), (b) and (c).
This finishes the construction of $X$ and $(m_l)_{l<\omega}$.

Step 2.
We use the Halpern-L\"{a}uchli theorem to shrink $X$ to an $\pmb{F_{\max}}$-tree $T$ such that
the set $\{b\in  \mathcal{F}_{|a|+1} : b\subseteq T \}$ is either included in  or disjoint from $\mathcal{O}$.
We define a coloring $$c :
\bigcup_{l< \omega}
\prod_{i\leq d}X_{v_i}\cap \omega^{m_l}\longrightarrow 2
$$
as follows:
Let $c(v_0, v_1,\dots, v_d)=1$.
Suppose $0<l<\omega$ and $(t_i)_{i\leq d}\in X_{v_i}\cap \omega^{m_l}$. Since $S$ is a skew tree, so is $X$.
Let
$$
sl(t_0)=\max\{|t| : t\sqsubseteq t_0 \ \mbox{and   succ}_X(t)>1\}+1.
$$
Then $(t_i\upharpoonright sl(t_0))_{1\leq i\leq d}$ is a tuple of
nephews for $t_0\upharpoonright sl(t_0)-1$. Let
$$\begin{array}{ll}
c(t_0, t_1,\dots, t_d)=\left\{\begin{array}{ll}
1,&\Big\{b_{\pmb{F}, (t_i\upharpoonright sl(t_0))} : \pmb{F}\in \big(\mathop{}_{ \ \ \ \ \ \pmb{F_{j}}}^{\pmb{F^X_{t_{0}\upharpoonright sl(t_0)-1}}}\big)
\Big\}\subseteq \mathcal{O},\\
0, &\Big\{b_{\pmb{F}, (t_i\upharpoonright sl(t_0))} : \pmb{F}\in \big(\mathop{}_{ \ \ \ \ \ \pmb{F_{j}}}^{\pmb{F^X_{t_{0}\upharpoonright sl(t_0)-1}}}\big)\Big\}
\cap  \mathcal{O}=\emptyset.\end{array}\right.
\end{array}$$ By (c), this colouring is well-defined. By Theorem 4.3,
there are a strictly increasing sequence $(n_j)_{j<\omega}\subseteq (m_l)_{l<\omega}$ and  strong
subtrees $$Y_i\subseteq \bigcup_{l< \omega}
X_{v_i}\cap \omega^{m_l}$$ witnessed by $(n_j)_{j<\omega}$ such that  $c$ is monochromatic on
$\bigcup_{j< \omega}
\prod_{i\leq d}Y_{i}\cap \omega^{n_j}$.
Let $T$ be the downward closure of $\bigcup_{i\leq d}Y_i$.

\begin{claim}
 If $l< \omega$, $i\leq d$, $s\in Y_i\cap \omega^{m_l}$ and $t$ corresponds to $s$ as in (a) and (b), then $t\in  T$
and $\pmb{F_t^{T}}=\pmb{F_t^X}$.
\end{claim}

\begin{proof}
Since $s\in Y_i\cap \omega^{m_l}$, there is  some $p< \omega$ such that $n_p= m_l$. As $Y_i$ is a strong
subtree of $\bigcup_{l^\prime< \omega}X_{v_i}\cap \omega^{m_{l^\prime}}$, every immediate successor of $s$ in $\bigcup_{l^\prime< \omega}X_{v_i}\cap \omega^{m_{l^\prime}}$
 has exactly one extension in $Y_i\cap \omega^{n_{p+1}}$. It follows from the construction of $X$  that every node in succ$_X(t)$ has exactly one
extension in $X\cap \omega^{m_{l+1}}$.  Moreover, the immediate successors of $s$ in  $\bigcup_{l^\prime< \omega}X_{v_i}\cap \omega^{m_{l^\prime}}$
are precisely the
extensions in $X\cap \omega^{m_{l+1}}$ of nodes in succ$_X(t)$. Then $t$ is in  the downward closure of $Y_i$, and  thus $t\in T$. So $\pmb{F_t^{T}}=\pmb{F_t^X}$.
\end{proof}

Then it is straightforward to check that $T$ is an $\pmb{F_{\max}}$-tree such that
the set $\{b\in \mathcal{F}_{|a|+1} : b\subseteq  T\}$ is either included in or disjoint from $\mathcal{O}$. By the
observation before this lemma, we can further shrink $T$ to $U\in  \mathcal{F}_{\infty}$ satisfying the
conclusion of the lemma.\end{proof}

\begin{theorem}\label{thm.FclasstRs}
For each   Fra\"{\i}ss\'{e}
  class of finite ordered binary relational structures with the Ramsey property and each type $\tau$,
the space $(\mathcal{F}_\infty(\tau), \leq, r)$
 is a topological Ramsey space.
\end{theorem}



\vspace*{0.2cm}

\section{Finite big Ramsey degrees
for ordered binary relational universal inverse limit structures}

\vspace*{0.4cm}

Let $\mathcal{K}$ be a Fra\"{\i}ss\'{e} class of finite ordered binary relational structures
satisfying the Ramsey property, with  signature $L=\{<,R_0,\dots, R_{k-1}\}$ where  each $R_i$, $i<k$, is a binary relation.
 In this section,
 we prove that  for each such $\mathcal{K}$, the universal inverse limit structure has finite big Ramsey degrees under finite Baire-measurable colorings.
 The proofs in this section are straightforward via
the
  topological Ramsey spaces
 from Theorem \ref{thm.FclasstRs}
 (which is based on work of Zheng in \cite{YY18})
 and the work of Huber-Geschke-Kojman on
inverse limits of finite ordered graphs  in \cite{SH19}.

\begin{definition}
Let $\mathcal{K}$ be a Fra\"{\i}ss\'{e} class of finite ordered binary relational structures with the Ramsey property.  A {\em universal inverse limit} of finite ordered structures in $\mathcal{K}$
 is a triple $\pmb{G}=\langle G, <^{\pmb{G}}, R^{\pmb{G}}_0,\dots, R^{\pmb{G}}_{k-1}\rangle$, such
that the following conditions hold.

1. $G$ is a compact subset of $\mathbb{R}\backslash\mathbb{Q}$ without isolated points,
$ <^{\pmb{G}}$ is the
restriction of the standard order on $\mathbb{R}$ to $G$, and
 $ R^{\pmb{G}}_i\subseteq [G]
^2$ for each $i<k$.

2. For every pair of distinct elements $u, v\in  G$, there is a partition of $G$ to finitely
many closed intervals such that

(a) $u, v$ belong to different intervals from the partition;

(b) for every interval $I$ in the partition, for all $x\in G\backslash I$ and for all $y, z\in I$,
$(x, y)\in R^{\pmb{G}}_i$ if and only if $(x, z)\in R^{\pmb{G}}_i$, for each $i<k$.

3. Every nonempty open interval of $G$ contains induced copies of all finite ordered
structures in $\mathcal{K}$.
\end{definition}

For
every $\pmb{F_{\max}}$-tree $T$, it can be seen  from Definition
\ref{defn.2.6}
that $\pmb{F(T)}$ is a universal inverse limit
structure.
So it follows from the universality that  we can consider colourings of finite induced substructures of $\pmb{F(T)}$.

\begin{definition}
Let $\tau$ be a type and $\pmb{H}\in \tau$. The $\mathcal{F}_\infty$-envelope of $\pmb{H}$ is
$$\mathcal{C}_{\pmb{H}} =\{U\in \mathcal{F}_\infty : (\exists \ l)(r_l(U)=\dn \pmb{H}^\vee)\},$$
where $\dn H^\vee=\{a\in \omega^{<\omega} : (\exists \ x\in H)(a\subseteq x\upharpoonright (\triangle(\pmb{H})+1))\}$.
\end{definition}

\begin{lemma}\label{lem.5.3}
Let $\tau$ be a type and  $T\in \mathcal{F}_\infty$. Define a map
$c_1 : [\emptyset, T]\longrightarrow \big\{\pmb{H}^\vee : \pmb{H}\in \big(\mathop{}_{ \ \ \tau}^{\pmb{F(T)}}\big)\big\}$ as follows:
$$\forall \ U\in [\emptyset, T],\ \mbox{if} \ U\in \mathcal{C}_{\pmb{H}},  \ c_1(U)=\pmb{H}^\vee.$$
Then $c_1$ is well-defined and continuous, where we equip the range with the  discrete topology.
\end{lemma}

\begin{proof}
Let $m+1$ be the number of elements for each $\pmb{H}\in \tau$. Then $\dn \pmb{H}^\vee$ has $m$ splitting nodes. Thus
$$\forall \ U\in \mathcal{F}_\infty,  \forall \ l,
(r_l(U)=\dn \pmb{H}^\vee\Rightarrow l=m).$$ Let $U, V\in  [\emptyset, T]$. Then there are $\pmb{H}, \pmb{K}\in \tau$ such that $U\in \mathcal{C}_{\pmb{H}}$ and
$V\in \mathcal{C}_{\pmb{K}}$.  If $U=V$, then $\dn \pmb{H}^\vee=r_m(U)=r_m(V)=\dn \pmb{K}^\vee$, and thus $\pmb{H}^\vee=\pmb{K}^\vee$. So $c_1$ is well-defined.

Suppose that $\pmb{H}\in \tau$  and $U\in (c_1)^{-1}(\pmb{H}^\vee)$. We have that $U\in \mathcal{C}_{\pmb{H}}$. Then  the set $[m, U]$ is an open set containing $U$ and $[m, U]\subseteq (c_1)^{-1}(\pmb{H}^\vee)$. Thus $c_1$ is continuous.\end{proof}

We equip $\omega^{\omega}$ with the first-difference metric topology, which has basic open sets of
the form $[s]=\{x\in \omega^{\omega} : s\subseteq x\}$ for $s\in  \omega^{<\omega}$. For $n\in  \omega$, let $[\pmb{F_{\max}}]^n$
denote the set of all induced substructures of $\pmb{F_{\max}}$ of size $n$.

\begin{definition}   For $n\geq 1$,  we define a topology on $[\pmb{F_{\max}}]^n$
as follows: A
set $\mathcal{U}\subseteq [\pmb{F_{\max}}]^n$
is open if for all $\pmb{H}\in  \mathcal{U}$,  there are open neighborhoods
$U_1, \dots, U_n$ of the elements of $H$ such that all $\pmb{H^\prime}\in [\pmb{F_{\max}}]^n$
that have exactly
one vertex in each $U_i$ are also in $\mathcal{U}$. This topology is separable and induced
by a complete metric. A coloring of $n$-tuples from $[\pmb{F_{\max}}]^n$ is continuous if it
is continuous with respect to this topology.
\end{definition}

\begin{lemma}
 Let $\tau$ be a type and $T\in \mathcal{F}_\infty$.  For every continuous coloring
$c : \big(\mathop{}_{ \ \ \tau}^{\pmb{F(T)}}\big)
\longrightarrow 2$,  there exists an $\pmb{F_{\max}}$-subtree $S$ of $T$ such that $c$ depends only on $\pmb{H}^\vee$, i.e., for $\pmb{H}, \pmb{K}\in
\big(\mathop{}_{ \ \ \tau}^{\pmb{F(S)}}\big)$, if $\pmb{H}^\vee=\pmb{K}^\vee$, then $c(\pmb{H})=c(\pmb{K})$.
\end{lemma}

\begin{proof}
For  $\pmb{H}\in \big(\mathop{}_{ \ \ \tau}^{\pmb{F(T)}}\big)$, by
definition of $\triangle(\pmb{H})$, the map $x\longmapsto x\upharpoonright (\triangle(\pmb{H})+1))$ is a bijection from
the universe  $H$ of  $\pmb{H}$ onto $\pmb{H}^\vee$.
Let $t_1,\dots, t_l$ denote
the elements of $\pmb{H}^\vee$. For all $\bar{x}=(x_1,\dots, x_l)\in [T_{t_1}]\times \cdot\cdot\cdot\times [T_{t_l}]$,
the induced substructure $\pmb{\overline{H}}$ of $\pmb{F(T)}$ on the set $\{x_1,\dots, x_l\}$ is isomorphic to $\pmb{H}$.
By the continuity of $c$, for all such $x$ there are open neighborhoods $x_1\in U^{\bar{x}}_{1},\dots,  x_l\in U^{\bar{x}}_{l}$
 such that for all $(y_1,\dots, y_l)\in  U^{\bar{x}}_{1}
\times\cdot\cdot\cdot\times U^{\bar{x}}_{l}$
for the induced
substructure $\pmb{H^\prime}$ of $\pmb{F(T)}$ on the vertices $y_1,\dots, y_l$, we have $c(\pmb{\overline{H}})=c(\pmb{H^\prime})$.

We may assume that the $U^{\bar{x}}_{i}$
are basic open sets, i.e., sets of the form $[T_r]$
for some $r\in T$. Since the space $[T_{t_1}]\times\cdot\cdot\cdot\times [T_{t_l}]$ is compact, there is a finite
set $A \subseteq
[T_{t_1}]\times\cdot\cdot\cdot\times [T_{t_l}]$ such that
$$[T_{t_1}]\times\cdot\cdot\cdot\times [T_{t_l}]=
\bigcup\limits_{\bar{x}\in A}\prod\limits_{i=1}^{l}U^{\bar{x}}_{i}.$$
Hence there is  $m\in \omega$ such that for all induced substructures $\pmb{H^\prime}$ of $\pmb{F(T)}$ with $\pmb{H^\prime}\upharpoonright (\triangle(\pmb{H})+1)=\pmb{H}^\vee$,  the color
$c(\pmb{H^\prime})$ only depends on $\pmb{H^\prime}\upharpoonright m$, where $m$ is  the maximal length of the $r$'s with
$[T_r]= U^{\bar{x}}_{i}$
for some $\bar{x}\in F$ and $i\in \{1,\dots, l\}$.

Since for each $m\in \omega$, there are only finitely many sets of the form $\pmb{H}\upharpoonright m$,
where $H\in \big(\mathop{}_{ \ \ \tau}^{\pmb{F(T)}}\big)$, there is a function $f : \omega\longrightarrow \omega$ such that for
every finite induced substructure $\pmb{H}$ of $\pmb{F(T)}$ with $\triangle(\pmb{H})+1=n$, the color $c(\pmb{H})$
only depends on $\pmb{H}\upharpoonright f(n)$.
Now let $S$ be an $\pmb{F_{\max}}$-subtree of $T$ such that
whenever $s\in S$ is a splitting node of $S$ of length $n$, then $S$ has no  splitting node  $t$ whose length is in the interval $(n, f(n)]$. Now for all $H\in \big(\mathop{}_{ \ \ \tau}^{\pmb{F(S)}}\big)$, the color $c(\pmb{H})$ only depends on $\pmb{H}^\vee$.\end{proof}

\begin{theorem}
 Let $T$ be an $\pmb{F_{\max}}$-tree. For every type $\tau$ of a finite induced substructure   of $\pmb{F_{\max}}$, and  every continuous coloring $c : \big(\mathop{}_{ \ \ \tau}^{\pmb{F(T)}}\big)
\longrightarrow 2$, there is an
$\pmb{F_{\max}}$-subtree $S$ of $T$ such that $c$ is monochromatic  on $\big(\mathop{}_{ \ \ \tau}^{\pmb{F(S)}}\big)$.
\end{theorem}

\begin{proof}We can shrink $T$ and assume
$T\in \mathcal{F}_\infty$.   By Lemma 5.5,  $c$ depends only on $\pmb{H}^\vee$.  We may think of $c$ as a map as follows:
$$c : \Big\{\pmb{H}^\vee : \pmb{H}\in \big(\mathop{}_{\tau}^{\pmb{F(T)}}\big)\Big\}\longrightarrow 2.$$
Define $\bar{c} : [\emptyset, T ]\longrightarrow 2$ by $\bar{c}=c\circ c_1$. By Lemma 5.3, $\bar{c}$ is also a continuous
map.
By Theorem \ref{thm.FclasstRs},
 there is some $S\leq T$ such that $\bar{c}$   is monochromatic on $[\emptyset, S]$.
Suppose that $\pmb{H}\in \big(\mathop{}_{ \ \ \tau}^{\pmb{F(S)}}\big)$.
Then the universe  $H$ of $\pmb{H}$ is contained in
$[S]$, so $\pmb{H}^\vee\subseteq S$.
Hence there is a
$U\in [\emptyset, S]$ such that $U\in  \mathcal{C}_{\pmb{H}}$,
 and thus,  $c(\pmb{H})=\bar{c}(U)$.
 Therefore $c$ is monochromatic  on $\big(\mathop{}_{ \ \ \tau}^{\pmb{F(S)}}\big)$.
 \end{proof}

The next  lemma  is a straightforward  extension of Lemma  3.8 in \cite{SH19}.

\begin{lemma}  Let $T$ be an $\pmb{F_{\max}}$-tree. For every type $\tau$ of a finite induced substructure   of $\pmb{F_{\max}}$, and every Baire-measurable coloring $c : \big(\mathop{}_{ \ \ \tau}^{\pmb{F(T)}}\big)
\longrightarrow 2$, there is an
$\pmb{F_{\max}}$-subtree $S$ of $T$ such that $c$ is  continuous  on $\big(\mathop{}_{ \ \ \tau}^{\pmb{F(S)}}\big)$.
\end{lemma}

\begin{proof}Since $c : \big(\mathop{}_{ \ \ \tau}^{\pmb{F(T)}}\big)
\longrightarrow 2$ is Baire-measurable,
$c^{-1}(0)$ and $c^{-1}(1)$ have the property of Baire. Then there exist open sets $U$, $V$ in $\big(\mathop{}_{ \ \ \tau}^{\pmb{F(T)}}\big)$ and meager sets $M$, $N$   in $\big(\mathop{}_{ \ \ \tau}^{\pmb{F(T)}}\big)$ such that
$$c^{-1}(0)=U\triangle M \ \mbox{and} \ c^{-1}(1)=V\triangle N,$$
where $\triangle$ denotes   the symmetric difference.  Let $(N_n)_{n\in \omega}$ be a sequence of
closed nowhere dense subsets of $\big(\mathop{}_{ \ \ \tau}^{\pmb{F(T)}}\big)$
such that $M\cup N\subseteq \bigcup_{n\in \omega}N_n$. We would like to construct an $\pmb{F_{\max}}$-subtree $S$ of $T$ such that $\big(\mathop{}_{ \ \ \tau}^{\pmb{F(S)}}\big)$ is disjoint from $\bigcup_{n\in \omega}N_n$. In this case, we have
$$c^{-1}(0)\cap \big(\mathop{}_{\tau}^{\pmb{F(S)}}\big)=U\cap \big(\mathop{}_{\tau}^{\pmb{F(S)}}\big) \ \mbox{and} \ c^{-1}(1)\cap \big(\mathop{}_{\tau}^{\pmb{F(S)}}\big)=V\cap \big(\mathop{}_{\tau}^{\pmb{F(S)}}\big).$$
It follows that $c$ is continuous on $\big(\mathop{}_{ \ \ \tau}^{\pmb{F(S)}}\big)$.
In order to find an $\pmb{F_{\max}}$-subtree $S$ that is disjoint from $\bigcup_{n\in \omega}N_n$, we construct a fusion sequence $(T_j)_{j\in \omega}$ of  $\pmb{F_{\max}}$-subtrees of $T$ with witness a strictly
increasing sequence $(m_j)_{j\in \omega}$ of natural numbers.
Put $S=\bigcap_{j\in \omega}T_j$. Then $S$ is an $\pmb{F_{\max}}$-subtree of $T$.

Suppose $T_j$ and $m_j$ have already been chosen. We assume that for all
$t\in T_j(m_j)$ and all $s\in T$ with $t\subseteq s$, we have $s\in T_j$. For a certain $t\in T_j(m_j)$, we have to  find a splitting node $s$ with  $t\subseteq s$
such that for  a certain finite ordered  structure $\pmb{H}$, $\pmb{H}$ embeds into $\pmb{F
^{T_{j+1}}_s}$.
 Since $T_j$ is
an $\pmb{F_{\max}}$-tree, there is $m> m_j$ and an extension  $s$ of  $t$ with $|s|
< m$ such that $\pmb{H}$ embeds into $\pmb{F^{T_j}_s}$.

Suppose that $\pmb{H}$ is a finite substructure of $\pmb{F(T_j)}$ of type $\tau$ such that $\triangle(\pmb{H})< m$.
We list elements of $H$ as $t_0, t_1,\dots, t_p$.
 The set $\pmb{H}\upharpoonright m$ determines an open
subset $O$ of $\big(\mathop{}_{ \ \ \tau}^{\pmb{F(T)}}\big)$.
Since  $\bigcup_{n\leq j}N_n$ is closed and nowhere dense in $\big(\mathop{}_{ \ \ \tau}^{\pmb{F(T)}}\big)$,
 $O$ contains a  nonempty open subset that is disjoint from  $\bigcup_{n\leq j}N_n$.
 It
follows that for $i\in \{0,1,\dots,p\}$,    $t_i\upharpoonright m$ has an extension $s_i\in T_j$ such that
the open subset of $\big(\mathop{}_{ \ \ \tau}^{\pmb{F(T)}}\big)$
determined by $s_0, s_1, \dots, s_p$
is disjoint from  $\bigcup_{n\leq j}N_n$.
We may assume that $s_0, s_1, \dots, s_p$ have the same length $m_{j+1}> m$.

Let $X\subseteq T_j(m_{j+1})$
 be a set that contains exactly one extension of every
element of $T_j(m_{j})$  and in particular the elements $s_0,s_1,\dots,s_p$. Let
$$T_{j+1}=\{t\in T_j : \exists s\in X ( s\subseteq t\vee t\subseteq s)
\}.$$
Then $T_{j+1}$ is an $\pmb{F_{\max}}$-tree.
Whenever $\pmb{H^\prime}$
is a finite substructure of $\pmb{F(T_{j+1})}$
 of type $\tau$ with
$\pmb{H^\prime}\upharpoonright m=\pmb{H}\upharpoonright m$, then $\pmb{H^\prime}\upharpoonright m_{j+1}=\{s_0,s_1,\dots,s_p\}$.
In particular, $\pmb{H^\prime}\notin  \bigcup_{n\leq j}N_n$.  This finishes the recursive definition of the sequences $(T_j)_{j\in \omega}$
and $(m_j)_{j\in \omega}$.

 Finally, let $S=\bigcap_{j\in \omega}T_j$.
One can check  that $S$ is an $\pmb{F_{\max}}$-tree. Let $n\in \omega$ and let $\pmb{H}$ be a
finite substructure of $\pmb{F(S)}$ of type $\tau$. Then there is $j\in \omega$ such that $\triangle(\pmb{H})< m_j$. We
can choose $j\geq n$.
Note that $S(m_j)=T(m_j)$.  Since $S\subseteq T_{j+1}$, by the construction of $T_{j+1}$, $\pmb{H}\notin
\bigcup_{n\leq j}N_n$.
In particular, $H\notin N_n$. This shows that
$\big(\mathop{}_{ \ \ \tau}^{\pmb{F(S)}}\big)$
is disjoint from $\bigcup_{n\in \omega}N_n$. It follows that $c$ is continuous on $\big(\mathop{}_{ \ \ \tau}^{\pmb{F(S)}}\big)$.
\end{proof}

\begin{theorem}\label{thm.5.8}
 Let $T$ be an $\pmb{F_{\max}}$-tree.
 For every type $\tau$ of a finite induced  substructure of
 $\pmb{F_{\max}}$
  and every Baire-measurable coloring $c :
\big(\mathop{}_{ \ \ \tau}^{\pmb{F(T)}}\big)\longrightarrow 2$, there is an $\pmb{F_{\max}}$-subtree $S$ of $T$ such that $c$ is constant on $\big(\mathop{}_{ \ \ \tau}^{\pmb{F(S)}}\big)$.
\end{theorem}

\begin{proof}
By Lemma 5.7, there is an $\pmb{F_{\max}}$-subtree $U$
 of $T$ such that $c$ is continuous on $\big(\mathop{}_{ \ \ \tau}^{\pmb{F(U)}}\big)$.
By Theorem 5.6, there is an $\pmb{F_{\max}}$-subtree $S$ of $U$
such
that $c$ is constant on $\big(\mathop{}_{ \ \ \tau}^{\pmb{F(S)}}\big)$.\end{proof}

\begin{theorem}\label{thm.5.9}
 Let $\mathcal{K}$ be a Fra\"{\i}ss\'{e} class,  in a finite   signature, of finite ordered binary relational structures with the Ramsey property.
  For every $\pmb{H}\in \mathcal{K}$,
   there is a finite number $T(\pmb{H}, \pmb{F_{\max}})$ such that the following holds:
   For every universal  inverse limit structure
$\pmb{G}$ and  for each finite Baire-measurable coloring of the set $\big(\mathop{}_{\pmb{H}}^{\pmb{G}}\big)$
of all copies of $\pmb{H}$ in $\pmb{G}$, there is a closed
copy $\pmb{G^\prime}$ of $\pmb{G}$  contained in
$\pmb{G}$ such that  the set $\big(\mathop{}_{\pmb{H}}^{\pmb{G^\prime}}\big)$
of all copies of $\pmb{H}$ in $\pmb{G^\prime}$ has no more than $T(\pmb{H}, \pmb{F_{\max}})$ colors.
In particular,  $T(\pmb{H}, \pmb{F_{\max}})$ is at most the number of types associated to $\pmb{H}$.
\end{theorem}

\begin{proof}
We list all types for $\pmb{H}$ as $\tau_0, \tau_1,\dots, \tau_{m-1}$. Since $\pmb{G}$ is a universal  inverse limit structure, $\pmb{F_{\max}}$ embeds
continuously into it. Let $c$ be a finite Baire-measurable coloring of the set $\big(\mathop{}_{\pmb{H}}^{\pmb{G}}\big)$. Now work with the tree $T_{\max}$ coding $\pmb{F_{\max}}$.
Iterating Theorem \ref{thm.5.8},  there are  $\pmb{F_{\max}}$-trees $T_{\max}\geq T_{\tau_0}\geq \cdots \geq T_{\tau_{m-1}}$
so
that for each $i< m$, $c$ is constant on $\big(\mathop{}_{ \ \ \tau_i}^{\pmb{F(T_{\tau_i})}}\big)$.
We take $\pmb{G^\prime}=\pmb{F(T_{\tau_{m-1}}})$. Then $\pmb{G^\prime}$ is a closed
copy of $\pmb{G}$ contained in $\pmb{G}$,
and the set $\big(\mathop{}_{\pmb{H}}^{\pmb{G^\prime}}\big)$
of all copies of $\pmb{H}$ in $\pmb{G^\prime}$ has no more than $m$ colors.
\end{proof}

\vspace*{0.1cm}

\section{Exact big Ramsey degrees for some ordered binary  relational  inverse limit structures}

\vspace*{0.2cm}

In this section, we find the exact big Ramsey degrees in the
inverse limit structures  $\pmb{F_{\max}}$ of  the following Fra\"{\i}ss\'{e} classes in an ordered binary relational signature:
Free amalgamation classes, the class of
 finite ordered tournaments $\mathcal{OT}$,
  and the class of  finite  partial orders with a linear extension $\mathcal{OPO}$.
We shall do so by first  showing in Lemmas
\ref{lem.6.1},
\ref{lem.6.4},
and
\ref{lem.6.7}
that for any finite substructure $\pmb{H}$ of
$\pmb{F_{\max}}$, there is a larger finite structure $\pmb{\overline{H}}$  containing $\pmb{H}$ as an induced substructure such that
any copy of $\pmb{\overline{H}}$ in the inverse limit structure $\pmb{F_{\max}}$ must have exactly one meet in the tree $T_{\max}$.
Then in Theorem
\ref{thm.main},
we shall prove by induction on number of splitting nodes that each type persists in any  subcopy of $\pmb{F_{\max}}$.
This proves  that the exact big Ramsey degree for a given finite  substructure $\pmb{H}$ of $\pmb{F_{\max}}$
is exactly the number of types $\tau$ representing
$\pmb{H}$ in the tree $T_{\max}$.

Given a structure $\pmb{G}\le \pmb{F_{\max}}$, let $T_{\pmb{G}}=\{x\upharpoonright n : x\in \pmb{G}, n\in \omega\}$. Then $T_{\pmb{G}}$ is a subtree of  $T_{\max}$.
Given a tree $T$,  its {\em stem},  denoted stem$(T)$,  is the minimal  splitting node in $T$.

\begin{lemma}\label{lem.6.1}
Let $\mathcal{K}$ be any
\Fraisse\ class with free amalgamation
in an ordered binary relational signature.
Then for each $\pmb{H}\in  \mathcal{K}$,
there is  a  structure $\pmb{\overline{H}}\in \mathcal{K}$ containing a copy of $\pmb{H}$, where $\pmb{\overline{H}}$  has the following property:
 Given a universal  inverse limit structure
 $\pmb{G}$ for $\mathcal{K}$ contained in
 $\pmb{F_{\max}}$, every  copy $\pmb{\overline{I}}$ of $\pmb{\overline{H}}$ in
 $\pmb{G}$ has induced a subtree $T_{\pmb{\overline{I}}}$  of $T_{\pmb{G}}$ such that
 the type of $T_{\pmb{\overline{I}}}$ has exactly one splitting node.
 It follows that the
  immediate successors of stem$(T_{\pmb{\overline{I}}})$ in $T_{\pmb{\overline{I}}}$  have a copy of $\pmb{\overline{I}}$.
\end{lemma}

\begin{proof}
Let $\mathcal{K}$ be
as in the hypotheses,
let  $\pmb{G}$ be the universal inverse limit for $\mathcal{K}$ contained in
 $\pmb{F_{\max}}$, and fix  $\pmb{H}$  a finite substructure of
$\pmb{G}$.
Then $\pmb{H}$  is in $\mathcal{K}$.
Let $m$ be the size of the universe of $\pmb{H}$.
We construct a finite ordered structure
$\pmb{\overline{H}}\in \mathcal{K}$ of size $2m+1$,
  containing a copy of  $\pmb{H}$ as a substructure on the odd indexed vertices,
as follows.
Let $R$ denote the
 binary relation symbol $R_0$ in the signature of
$\mathcal{K}$.
\begin{enumerate}
\item
 Let $\overline{H}=\{v_0,v_1,v_2, \dots ,  v_{2m}\}$.
\item
$\pmb{\overline{H}}\upharpoonright\{v_1, v_3,\dots, v_{2m-1}\}$ is isomorphic to $\pmb{H}$.
\item
For $i\in \{0, 2,\dots, 2m-2\}$,
$R^{\pmb{\overline{H}}}(v_i, v_{i+2})$ holds.
If $R^{\pmb{\overline{H}}}$ is a symmetric relation, then also
$R^{\pmb{\overline{H}}}(v_{i+2}, v_i)$ holds;
otherwise, $\neg R^{\pmb{\overline{H}}}(v_{i+2}, v_i)$ holds.
\item
No other relations are added to $\pmb{\overline{H}}$.
\end{enumerate}

Then $\pmb{\overline{H}}$ contains a copy of $\pmb{H}$.  Now we check that $\pmb{\overline{H}}$ satisfies the property in this lemma.
Let  $\pmb{\overline{I}}$ be a copy of $\pmb{\overline{H}}$ in $\pmb{G}$ with $\overline{I}=\{u_0,u_1,u_2, \dots,  u_{2m}\}$.

\begin{claim}\label{claim.6.2}
Let $\pmb{J}$  be the  induced substructure
of $\pmb{\overline{I}}$ on  universe $J=\{u_0,u_2, \dots , u_{2m}\}$.
Then the associated subtree  $T_{\pmb{J}}$  of  $T_{\pmb{G}}$  has exactly one  splitting node.
\end{claim}

\begin{proof}
Without loss of generality, it suffices to prove that $u_0\cap u_2=u_2\cap u_4$.
 Assume to the contrary  that $u_0\cap u_2\neq u_2\cap u_4$. Then either
$u_0\cap u_2\subset u_2\cap u_4$ or else $u_2\cap u_4\subset u_0\cap u_2$,
where $\subset$ denotes proper subset.
If
$u_0\cap u_2\subset u_2\cap u_4$
(see Figure 20),
then
$u_0\cap u_2= u_0\cap u_4$;
let  $s$ denote this node and let $l-1$ denote its  length.
Then $u_2\upharpoonright l= u_4\upharpoonright l$ is a successor of $s$,
and
the  relation $R(u_0\upharpoonright l ,u_4\upharpoonright l)$
 holds   in $T_{\max}$ since  $R(u_0\upharpoonright l, u_2\upharpoonright l)$ holds in $T_{\max}$.
Hence,
$R^{\pmb{I}}(u_0, u_4)$ holds.
Similarly, if $u_2\cap u_4\subset u_0\cap u_2$,
 then $u_0\cap u_4=u_2\cap u_4$, and it follows that
 $R^{\pmb{I}}(u_0, u_4)$ holds (see Figure 20).
But this contradicts the fact that  $R^{\pmb{I}}(u_0,u_4)$ does not hold, by (3) in the defnition
of $\pmb{\overline{H}}$ above.
Therefore, it must be the case that
$u_0\cap u_2= u_2\cap u_4$.
Thus, every copy of $\pmb{J}$ has induced subtree $T_{\pmb{J}}$   in $T_{\pmb{G}}$  with exactly one  splitting node.
\end{proof}

\begin{center}
\centering
\includegraphics[totalheight=2.0in]{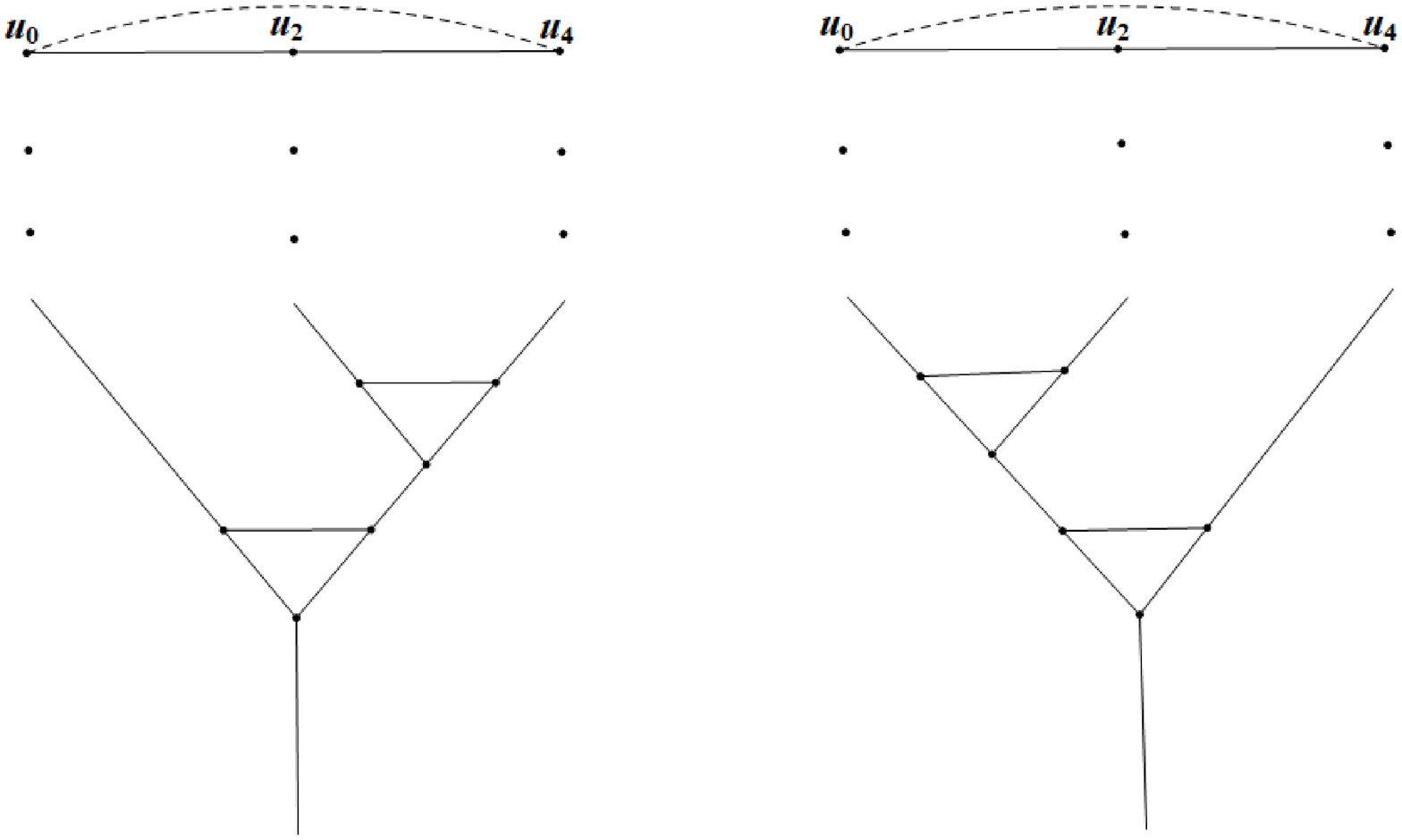}\\Figure 20
\end{center}

\begin{claim}\label{claim.6.3}
The subtree  $T_{\pmb{\overline{I}}}$  of  $T_{\pmb{G}}$
induced by  $\pmb{\overline{I}}$ has exactly one splitting node.
\end{claim}

\begin{proof}
By Claim \ref{claim.6.2}, without loss of generality, it suffices to prove that $u_0\cap u_1=u_1\cap u_2$. Assume that $u_0\cap u_1\neq u_1\cap u_2$.
 Then either
$u_0\cap u_1\subset u_1\cap u_2$ or $u_1\cap u_2\subset u_0\cap u_1$.
Since $u_0<u_1<u_2$ in the linear order $<$ on  the universe of
$\pmb{F_{\max}}$,
it follows that the set $\{u_0,u_1,u_2\}$ has type  equal to one of the two types  in Figure 21 (the solid lines).
Since $R^{\pmb{\overline{I}}}(u_0, u_2)$ holds,
at least one of
 $R^{\pmb{\overline{I}}}(u_0, u_1)$ or
 $R^{\pmb{\overline{I}}}(u_1, u_2)$
 holds (the dashed lines in Figure 21).
 But this contradicts the fact that $ \neg R^{\pmb{\overline{I}}}(u_0, u_1)$ and
  $\neg R^{\pmb{\overline{I}}}(u_1, u_2)$  hold, by (4) in the definition of $\pmb{\overline{H}}$.
  Thus, the induced subtree $T_{\pmb{\overline{I}}}$ has exactly one
   splitting node.\end{proof}

\begin{center}
\centering
\includegraphics[totalheight=2.0in]{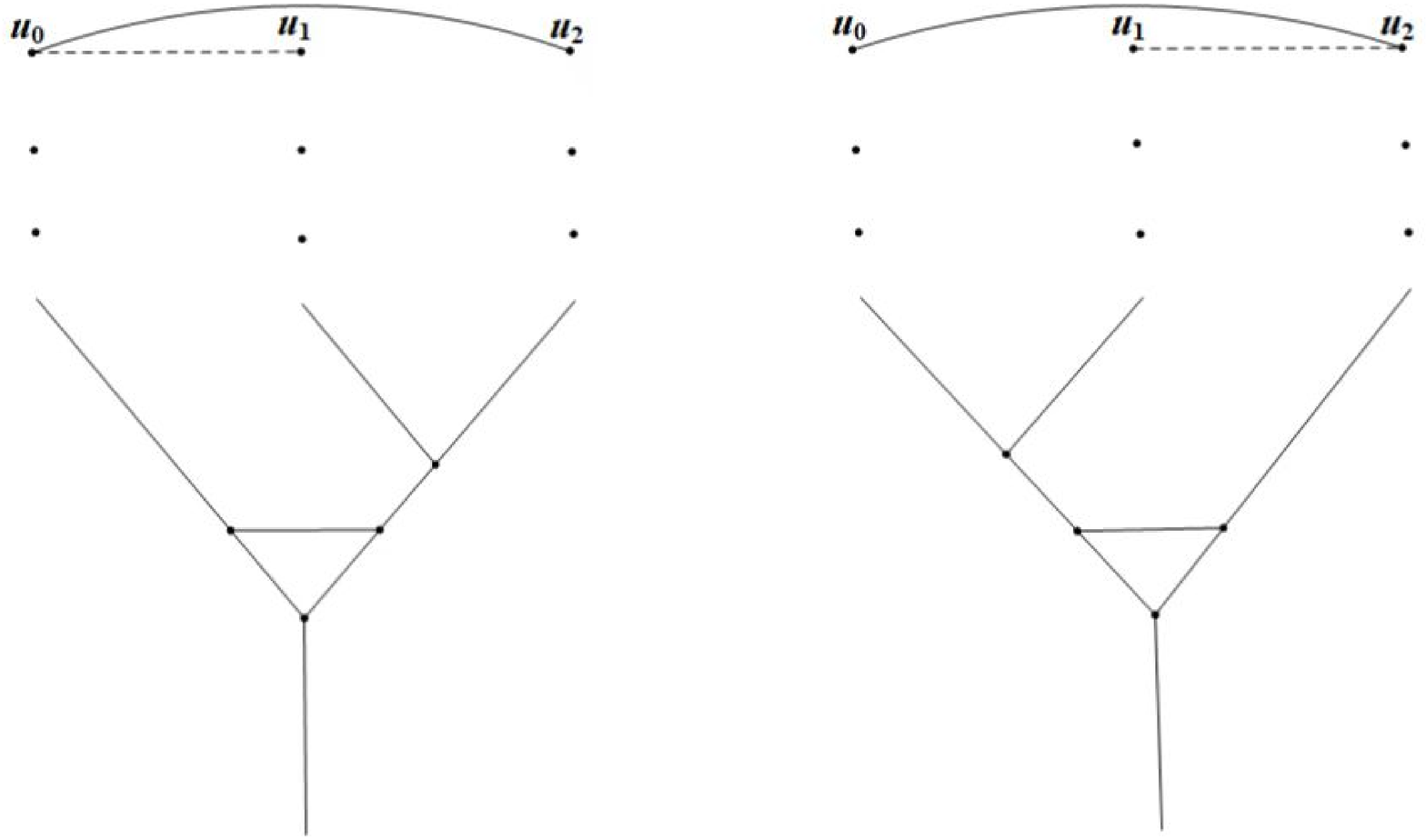}\\Figure 21
\end{center}

By Claim \ref{claim.6.3}, the induced subtree $T_{\pmb{\overline{I}}}$  of $T_{\pmb{G}}$
has exactly one splitting node, namely its stem.
It follows from the definition of the relations on
$\pmb{F_{\max}}$ that
the immediate successors of stem$(T_{\pmb{\overline{I}}})$ in $T_{\pmb{\overline{I}}}$ is isomorphic to
$\pmb{\overline{H}}$.
\end{proof}

Next, we prove a similar lemma for  the  class of finite ordered tournaments, $\mathcal{OT}$.
The proof is similar to the previous lemma, the main difference being that the construction of $\pmb{\overline{H}}$ must take into account  the fact that any two vertices of a tournament must have some directed  edge relation between them.

\begin{lemma}\label{lem.6.4}
 Let $\mathcal{K}$ be $\mathcal{OT}$.
    Then for each $\pmb{H}\in  \mathcal{K}$,
there is  a structure  $\pmb{\overline{H}}\in \mathcal{K}$ containing a copy of $\pmb{H}$
and $\pmb{\overline{H}}$  has the following property:
 Given a universal  inverse limit structure
 $\pmb{G}$ for $\mathcal{K}$
 contained in $\pmb{F_{\max}}$,
 every  copy $\pmb{\overline{I}}$ of $\pmb{\overline{H}}$ in
 $\pmb{G}$ has induced a subtree $T_{\pmb{\overline{I}}}$  of $T_{\pmb{G}}$ such that
 the type of $T_{\pmb{\overline{I}}}$  has exactly one splitting node.
 It follows that the
 immediate successors of stem$(T_{\pmb{\overline{I}}})$ in $T_{\pmb{\overline{I}}}$  have a copy of $\pmb{\overline{I}}$.
\end{lemma}

\begin{proof}
Fix any $\pmb{H}\in  \mathcal{K}$, and let
$m$ be the size of the universe of
 $\pmb{H}$.
 Recall that the relation $R$ here is a directed edge.
We construct a finite ordered structure
$\pmb{\overline{H}}\in \mathcal{K}$ containing a copy of
$\pmb{H}$ as an induced substructure as follows:
\begin{enumerate}
\item
Let $\overline{H}=\{v_0,v_1,v_2, \dots , v_{2m}\}$.
\item
$\pmb{\overline{H}}\upharpoonright\{v_1, v_3,\dots, v_{2m-1}\}$ is isomorphic to $\pmb{H}$.
\item
 For $i, j\in \{0, 2,\dots, 2m\}$ with $i< j$,  if  $j=i+2$
 then
$R^{\pmb{\overline{H}}}(v_i, v_j)$ holds.
\item
 For $i, j\in \{0, 2,\dots, 2m\}$ with $i< j$, if  $j\neq i+2$
 then
$R^{\pmb{\overline{H}}}(v_j, v_i)$ holds.
\item
 For all $i\in \{0, 2,\dots, 2m\}$ and $j\in \{1, 3,\dots, 2m-1\}$,\ if $i< j$ then
 $R^{\pmb{\overline{H}}}(v_j,v_i)$ holds.
\item
 For all $i\in \{0, 2,\dots, 2m\}$ and $j\in \{1, 3,\dots, 2m-1\}$,\ if $j< i$ then
$R^{\pmb{\overline{H}}}(v_i,v_j)$ holds.
\end{enumerate}
By (2),  $\pmb{\overline{H}}$ contains a copy of $\pmb{H}$. Now we check that $\pmb{\overline{H}}$ satisfies the property in this lemma.
Let  $\pmb{\overline{I}}$ be  any copy of $\pmb{\overline{H}}$ in $\pmb{G}$, say with universe $\overline{I}=\{u_0,u_1,u_2, ..., u_{2m}\}$.

\begin{claim}\label{claim.6.5}
 Let $\pmb{J}$  be the  induced substructure
of $\pmb{\overline{I}}$ on universe $J=\{u_0,u_2, ..., u_{2m}\}$.
Then
the associated  subtree $T_{\pmb{J}}$ of $T_{\pmb{G}}$ has exactly one splitting node.
\end{claim}

\begin{center}
\centering
\includegraphics[totalheight=2.2in]{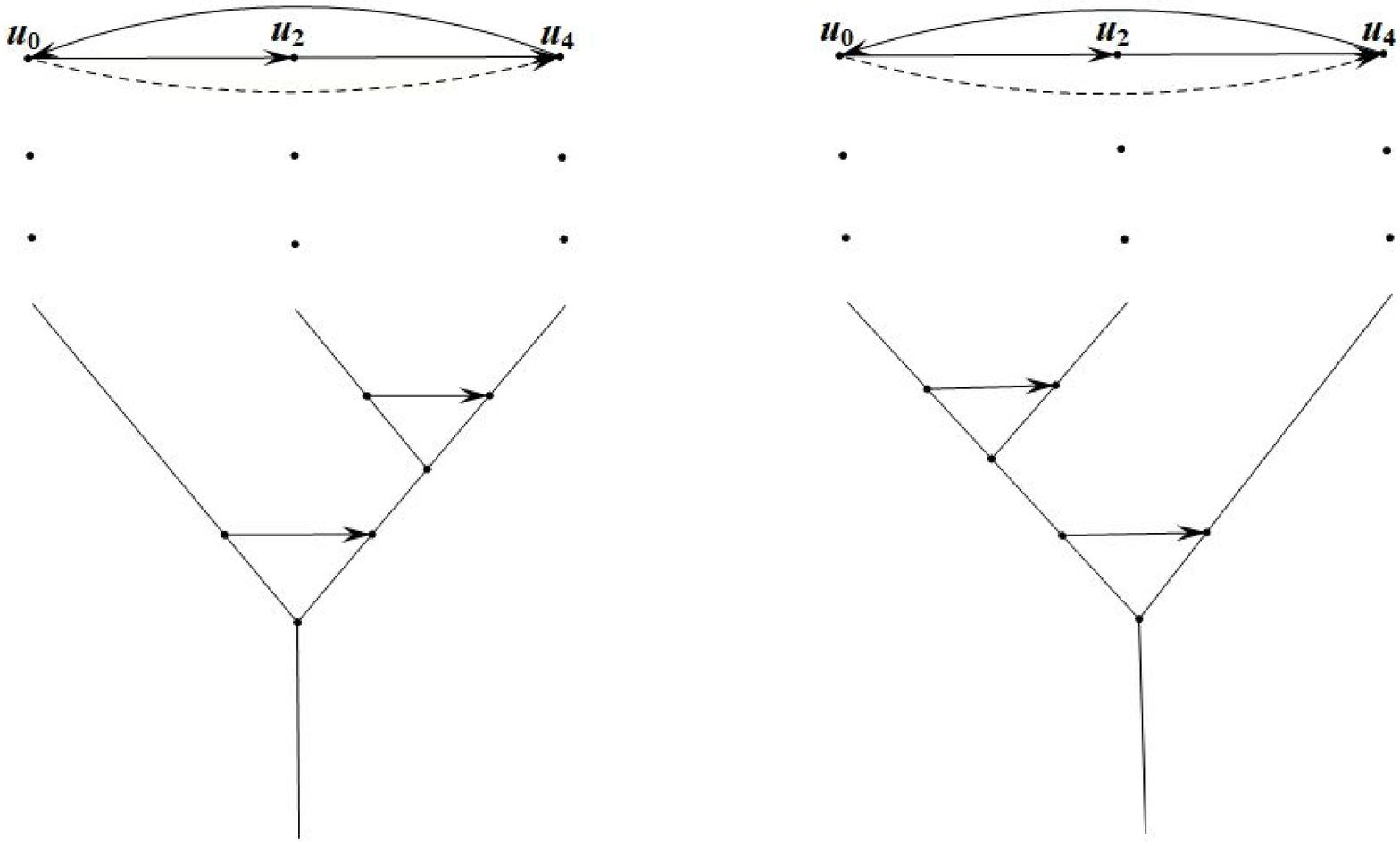}\\Figure 22
\end{center}

\begin{proof}
Without loss of generality, it suffices to prove that $u_0\cap u_2=u_2\cap u_4$.
Assume that $u_0\cap u_2\neq u_2\cap u_4$.
Then
$u_0\cap u_2\subset u_2\cap u_4$ or $u_2\cap u_4\subset u_0\cap u_2$.
Since $R^{\pmb{J}}(u_0, u_2)$  and
$R^{\pmb{J}}(u_2, u_4)$ hold,
it follows that
$R^{\pmb{J}}(u_0, u_4)$ holds (see the dashed arrows in  Figure 22).
This contradicts the fact that $R^{\pmb{J}}(u_4, u_0)$ holds, by (4) in the definition of $\pmb{\overline{H}}$ (solid arrows  from $u_4$ to $u_0$ in Figure 22).
Thus, $u_0\cap u_2= u_2\cap u_4$.
It follows that
$T_{\pmb{J}}$ has exactly one splitting node, namely its stem.
\end{proof}

\begin{claim}\label{claim.6.6}
The subtree $T_{\pmb{\overline{I}}}$  of $T_{\pmb{G}}$ induced by $\pmb{\overline{I}}$ has exactly one  splitting node.
 \end{claim}

\begin{proof}
By Claim \ref{claim.6.5}, without loss of generality, it suffices to prove that $u_0\cap u_1=u_1\cap u_2$.
Assume that $u_0\cap u_1\neq u_1\cap u_2$.
Then either
$u_0\cap u_1\subset u_1\cap u_2$ or
$u_1\cap u_2\subset u_0\cap u_1$.
Since $u_0<u_1<u_2$ in the linear order $<$ on the universe of
$\pmb{F_{\max}}$,
it follows that the set $\{u_0,u_1,u_2\}$ has type  equal to one of the two types  in Figure 23 (the solid arrows).
Since $R^{\pmb{\overline{I}}}(u_0, u_2)$,
either
  $ R^{\pmb{\overline{I}}}(u_1, u_2)$ or
$R^{\pmb{\overline{I}}}(u_0, u_1)$ holds
(see the dashed arrows in Figure 23).
This contradicts the facts that $ R^{\pmb{\overline{I}}}(u_1, u_0)$ and $R^{\pmb{\overline{I}}}(u_2, u_1)$ hold,  by (5) and (6) of the definition of
$\pmb{\overline{H}}$.
Thus, the induced subtree $T_{\pmb{\overline{I}}}$ has exactly one
   splitting node.
\end{proof}

\begin{center}
\centering
\includegraphics[totalheight=2.2in]{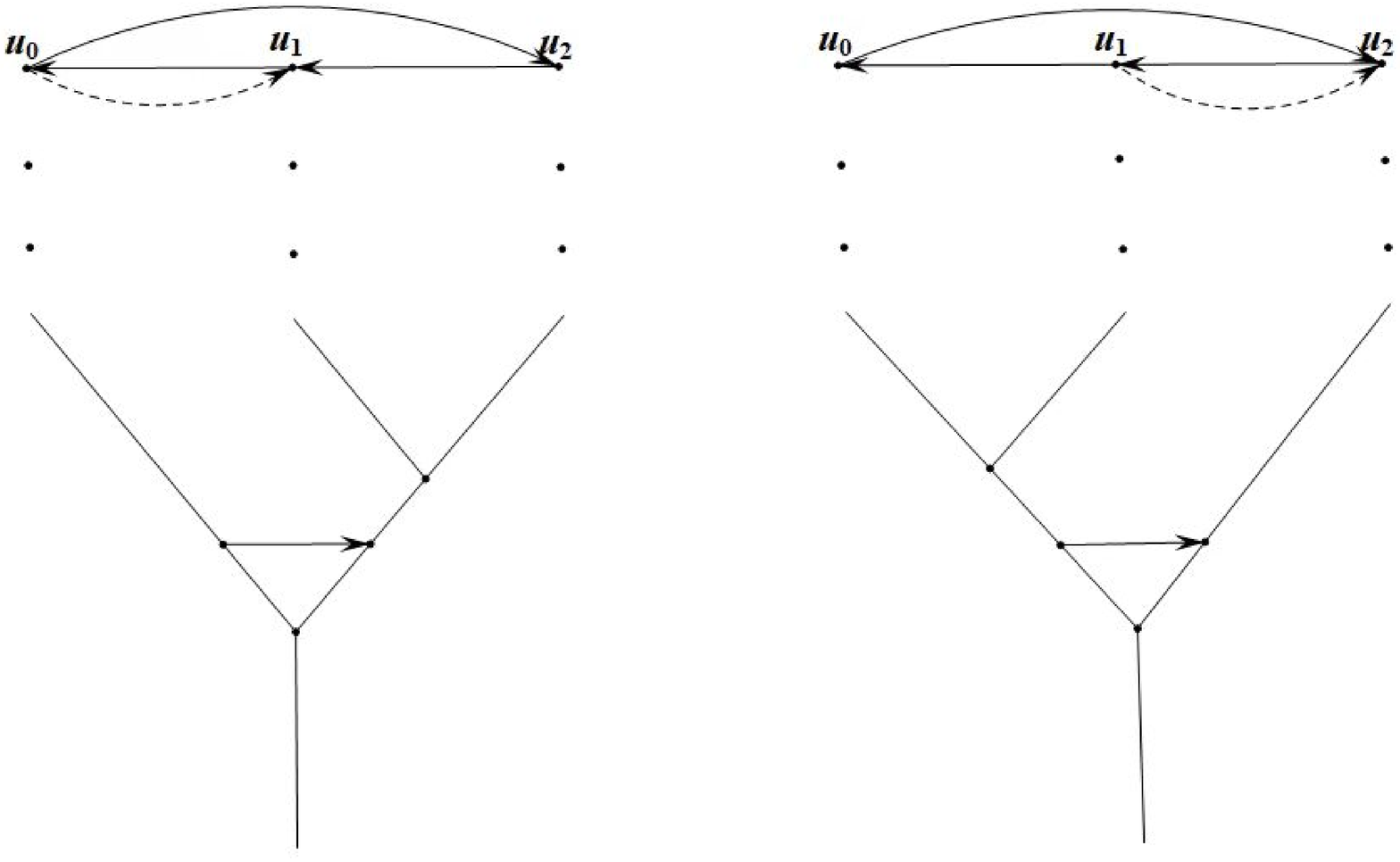}\\Figure 23
\end{center}

By Claim \ref{claim.6.6},   $\pmb{\overline{H}}$ satisfies the property in this lemma.
\end{proof}

Finally,
we prove a similar lemma for  the  class of finite
partial orders with a linear extension,
 $\mathcal{OPO}$.

\begin{lemma}\label{lem.6.7}
 Let $\mathcal{K}$ be $\mathcal{OPO}$.
    Then for each $\pmb{H}\in  \mathcal{K}$,
there is  a structure  $\pmb{\overline{H}}\in \mathcal{K}$ containing a copy of $\pmb{H}$
and $\pmb{\overline{H}}$  has the following property:
 Given a universal  inverse limit structure
 $\pmb{G}$ for $\mathcal{K}$
 contained in $\pmb{F_{\max}}$,
 every  copy $\pmb{\overline{I}}$ of $\pmb{\overline{H}}$ in
 $\pmb{G}$ has induced a subtree $T_{\pmb{\overline{I}}}$  of $T_{\pmb{G}}$ such that
 the type of $T_{\pmb{\overline{I}}}$  has exactly one splitting node.
 It follows that the
 immediate successors of stem$(T_{\pmb{\overline{I}}})$ in $T_{\pmb{\overline{I}}}$  have a copy of $\pmb{\overline{I}}$.
\end{lemma}

\begin{proof}
Fix any $\pmb{H}\in  \mathcal{K}$, and let
$m$ be the size of the universe of
 $\pmb{H}$.
If $\pmb{H}$ has universe of size one, then there is nothing to prove, so assume that the universe of  $\pmb{H}$ has size $m\ge 2$.
The relation $R$ here is a partial order, where $R(v,w)$ denotes that  $v$ is $R$-less than or equal to $w$.
We construct a finite ordered structure
$\pmb{\overline{H}}\in \mathcal{K}$ containing a copy of
$\pmb{H}$ as an induced substructure as follows:

\begin{enumerate}
\item
Let $\overline{H}=\{v_0,v_1,v_2, \dots , v_{2m+2}\}$.
\item
$\pmb{\overline{H}}\upharpoonright\{v_2, v_4,\dots, v_{2m}\}$ is isomorphic to $\pmb{H}$.

\item
$R(v_0,v_3)$  and $R(v_{2m-1},v_{2m+2})$ hold.
\item
 For  each $i \in \{1,3,\dots, 2m-3\}$,
$R^{\pmb{\overline{H}}}(v_i, v_{i+4})$ holds.
\item
The $R$-relations above are closed under transitivity of $R$, and no other $R$ relations are added.
\end{enumerate}
By (2),  $\pmb{\overline{H}}$ contains a copy of $\pmb{H}$.
Note that there are no $R$-relations between any vertices in $\{v_2,v_4,\dots,v_{2m}\}$ and any vertices $\{v_0,v_1,v_3,v_5,\dots,v_{2m+1},v_{2m}\}$.
For example, if $\pmb{H}\in \mathcal{OPO}$  with 3 vertices as in Figure 24, then $\pmb{\overline{H}}$  has 9 vertices as in Figure 25. The copy of $\pmb{H}$ in $\pmb{\overline{H}}$ is on vertices $\{v_2, v_4, v_6\}$.

\begin{center}
\centering
\includegraphics[totalheight=0.4in]{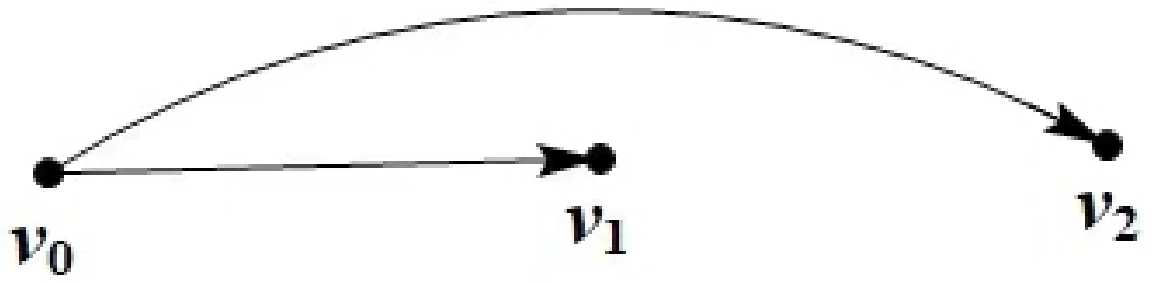}\\Figure 24
\end{center}

\begin{center}
\centering
\includegraphics[totalheight=1.2in]{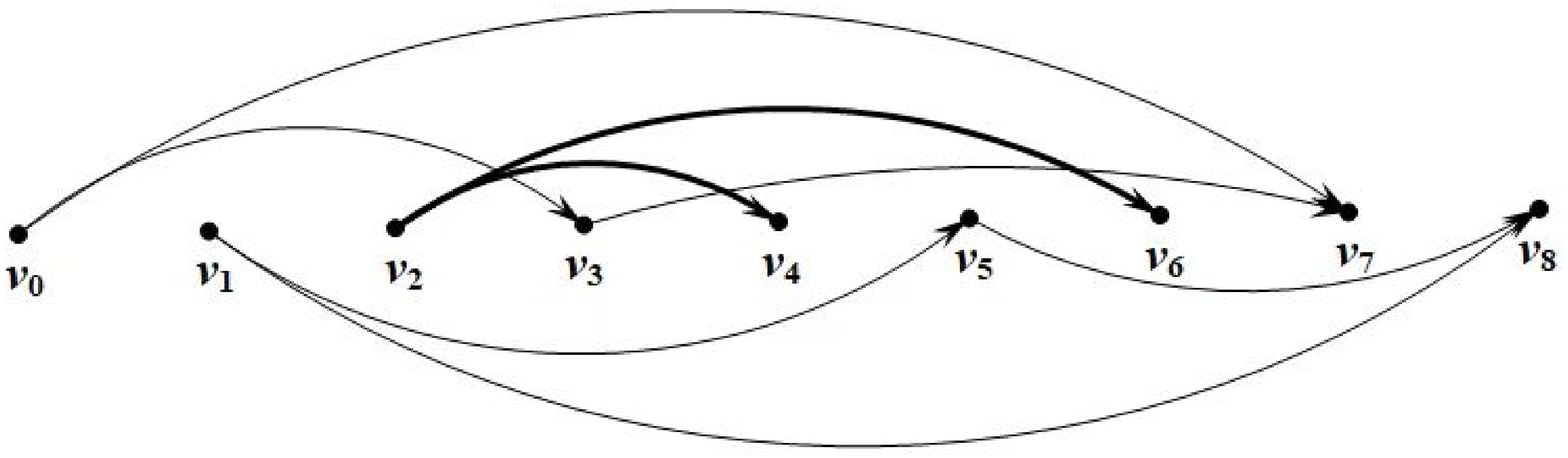}\\Figure 25
\end{center}

Now we check that $\pmb{\overline{H}}$ satisfies the property in this lemma.
Let  $\pmb{\overline{I}}$ be  any copy of $\pmb{\overline{H}}$ in $\pmb{G}$, say with universe $\overline{I}=\{u_0,u_1,u_2, \dots, u_{2m+2}\}$.

\begin{claim}\label{claim.6.8}
 Let $\pmb{J}$  be the  induced substructure
of $\pmb{\overline{I}}$ on universe $J=\{u_0,u_1,u_3, \dots,$ $u_{2m-1}, u_{2m+1}, u_{2m+2}\}$.
Then
the associated  subtree $T_{\pmb{J}}$ of $T_{\pmb{G}}$ has exactly one splitting node.
\end{claim}

\begin{proof}
It suffices to
prove that
any three successive vertices in $J$ have the same meet.
Without loss of generality, it suffices to  prove that
$u_0\cap u_1=u_1\cap u_3$,
as the same argument  shows that
for any  $0\le i\le m-2$,
$u_{2i+1}\cap  u_{2i+3} = u_{2i+3}\cap u_{2i+5}$, and that
$u_{2m-1}\cap u_{2m+1}= u_{2m+1}\cap u_{2m+2}$.

Assume that $u_0\cap u_1\neq u_1\cap u_3$.
Then
$u_0\cap u_1\subset u_1\cap u_3$ or $u_1\cap u_3\subset u_0\cap u_1$.
Note that  $R^{\pmb{J}}(u_0, u_3)$ holds
 and
$\neg R^{\pmb{J}}(u_0, u_1)$  and $\neg R^{\pmb{J}}(u_1, u_3)$ hold.
If
 $u_0\cap u_1\subset u_1\cap u_3$
then  $R^{\pmb{J}}(u_0,u_1)$, a contradiction.
If
$u_1\cap u_3\subset u_0\cap u_1$, then
$R^{\pmb{J}}(u_1,u_3)$,  also a contradiction
 (see   Figure 26).
Thus, $u_0\cap u_1= u_1\cap u_3$.
It follows that
$T_{\pmb{J}}$ has exactly one splitting node, namely its stem.
\end{proof}

\begin{center}
\centering
\includegraphics[totalheight=2in]{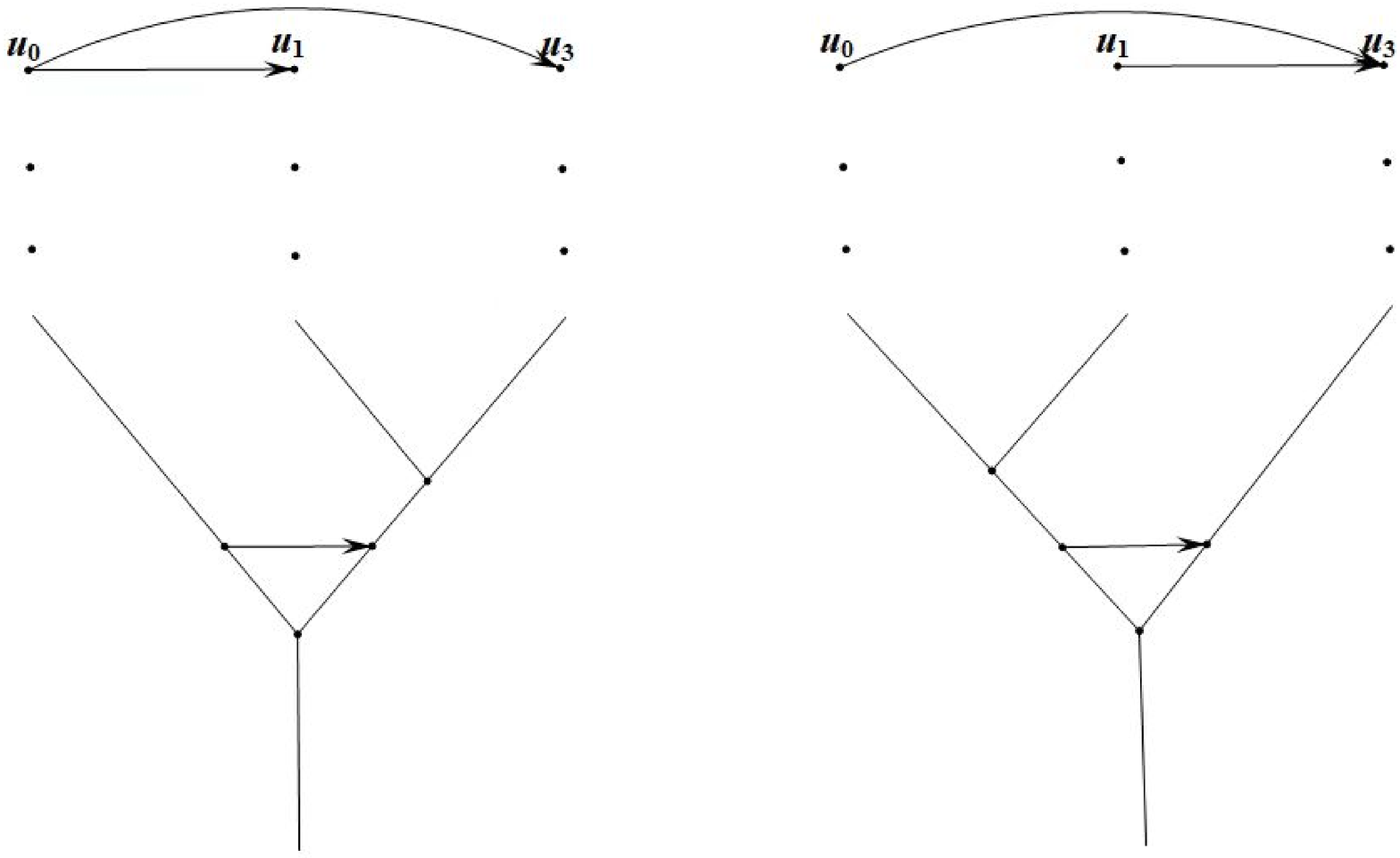}\\Figure 26
\end{center}

\begin{claim}\label{claim.6.9}
The subtree $T_{\pmb{\overline{I}}}$  of $T_{\pmb{G}}$ induced by $\pmb{\overline{I}}$ has exactly one  splitting node.
 \end{claim}

\begin{proof}
We will first prove that
for all $0\le i\le m-2$,
$u_{2i+1}\cap u_{2i+2}=u_{2i+2}\cap u_{2i+5}$,
and that
 $u_{2m-1}\cap u_{2m}=u_{2m}\cap u_{2m+2}$.
Since the argument is the same for each of these cases,
 it suffices to prove that  $u_1\cap u_2= u_2\cap u_5$.
In fact, the  same argument  as that  in Claim \ref{claim.6.8} applies here, for the structure
$\pmb{\overline{I}}$ restricted to the vertices $\{u_1,u_2, u_5\}$ is
isomorphic to the one in the proof of Claim \ref{claim.6.8}:
 $R^{\pmb{\overline{I}}}(u_1,u_5)$,
$\neg R^{\pmb{\overline{I}}}(u_1,u_2)$, and
$\neg R^{\pmb{\overline{I}}}(u_2,u_5)$.
Since the linear order $<$ extends $R^{\pmb{\overline{I}}}$,
$\neg R^{\pmb{\overline{I}}}(u_2,u_1)$,
$\neg R^{\pmb{\overline{I}}}(u_5,u_1)$, and
$\neg R^{\pmb{\overline{I}}}(u_5,u_2)$
all hold.
Therefore,  $u_1\cap u_2= u_2\cap u_5$ (see Figure 27).

Given  any three vertices $u_{2i+1}, u_{2i+2}, u_{2i+3}$, where $0\le i\le m-1$,
let $w=u_{2i+5}$ if $i<m-1$ and $w= u_{2m+2}$ if $i=m-1$.
By the above argument,
$$
u_{2i+1}\cap u_{2i+2}= u_{2i+2}\cap w=u_{2i+1}\cap w.
$$
By Claim \ref{claim.6.8},
$$
u_{2i+1}\cap u_{2i+3}=u_{2i+3}\cap w=u_{2i+1}\cap w.
$$
Therefore,
$$
u_{2i+1}\cap u_{2i+2} =
u_{2i+1}\cap w=
  u_{2i+2}\cap u_{2i+3}.
$$
Hence, also $u_{2i+1}\cap u_{2i+2}= u_{2i+1}\cap u_{2i+3}$.
Thus, any three successive vertices in
$\pmb{\overline{I}}$ have a common meet, meaning that
$\pmb{\overline{I}}$ has exactly one
   splitting node.
\end{proof}

\begin{center}
\centering
\includegraphics[totalheight=2.2in]{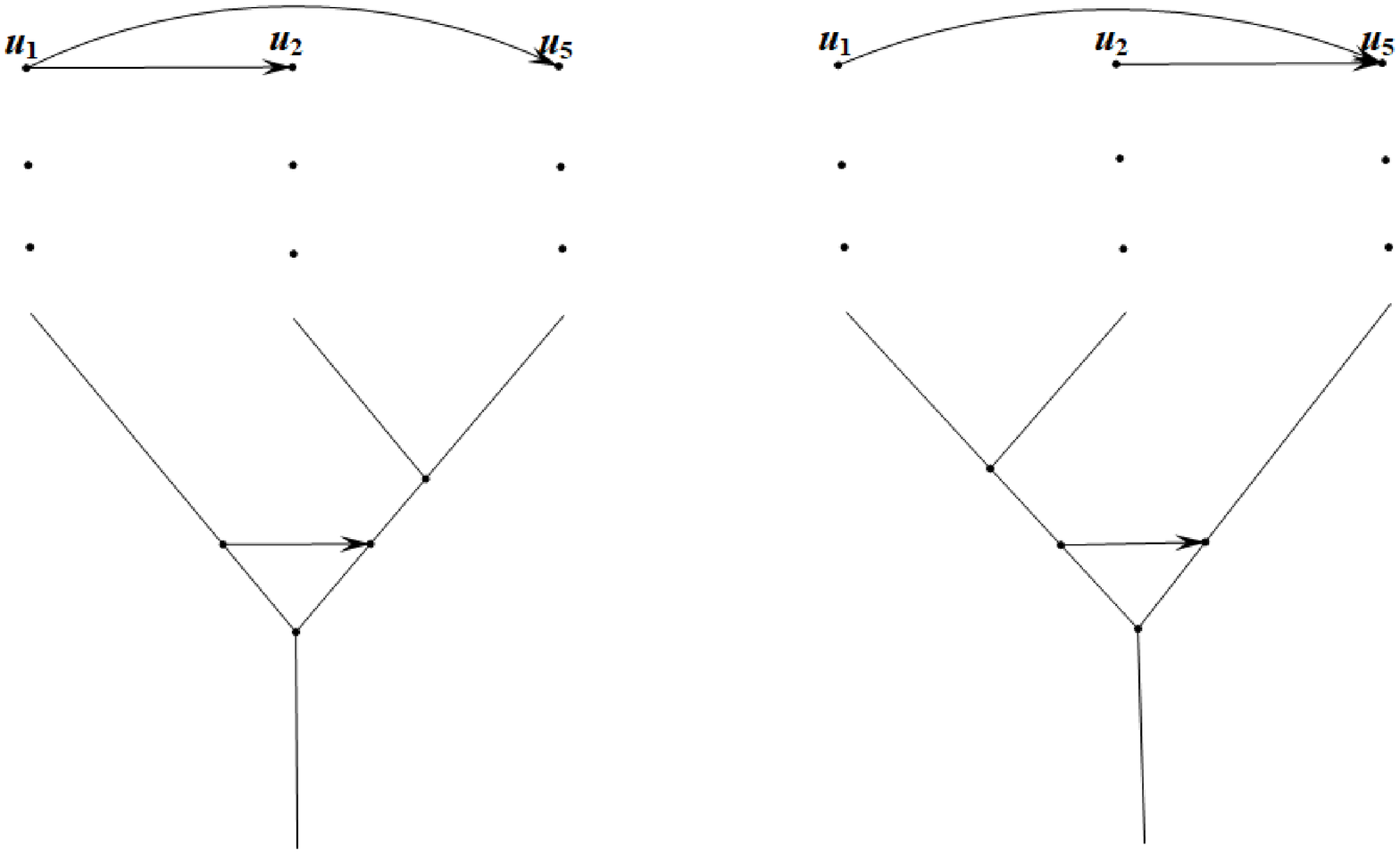}\\Figure 27
\end{center}

By Claim \ref{claim.6.9},   $\pmb{\overline{H}}$ satisfies the property in this lemma.
\end{proof}


The next theorem shows that the upper bounds proved in
Theorem \ref{thm.5.9} are exact for ordered binary  relational free amalgamation classes,  for ordered tournaments, and  for partial orders with a linear extension.

\begin{theorem}\label{thm.main}
Let $\mathcal{K}$ be any
\Fraisse\ class
in an ordered binary relational signature
such that either  $\mathcal{K}$ has  free amalgamation
or $\mathcal{K}$ is one of
$\mathcal{OT}$ or $\mathcal{OPO}$.
Let  $\pmb{G}$ be
 a  universal inverse limit structure  for  $\mathcal{K}$ contained in
 $\pmb{F_{\max}}$.
Then for each $\pmb{H}\in \mathcal{K}$, each type representing $\pmb{H}$  in $\pmb{G}$ persists in each closed subcopy  of $\pmb{G}$.
It follows that
 the  big Ramsey degree $T(\pmb{H}, \pmb{F_{\max}})$
for finite Baire-measurable colorings of
$\big(\mathop{}_{ \ \ \pmb{H}}^{\pmb{F_{\max}}}\big)$
is exactly the number of types in $T_{\max}$ representing a copy of $\pmb{H}$.
\end{theorem}

\begin{proof}
Let $\pmb{G}\subseteq \pmb{F_{\max}}$ be a universal  inverse limit structure for $\mathcal{K}$.
We will prove by induction on the number of splitting nodes that every type $\tau$ for each finite structure in $\mathcal{K}$  persists in $\pmb{G}$.

Suppose  $\tau$  is a type for a structure in $\mathcal{K}$ which has no  splitting nodes.
Then $\tau$ codes a single element,
so there is a copy of $\tau$ in $T_{\pmb{G}}$.

Now assume  that $n\ge 1$ and  for each  type $\tau$ with less than $n$ many splitting nodes,
the type $\tau$ appears in $T_{\pmb{G}}$.
For the induction step,
let $\tau$ be a type for a structure in $\mathcal{K}$  with exactly $n$ splitting nodes.
Let  $s$ denote  the  splitting node of longest length in $\tau$, and
 let $\sigma=\tau\upharpoonright |s|$.
 By the induction hypothesis, there is a copy of $\sigma$ in $T_{\pmb{G}}$.
  Then there is a subtree $U$ of $T_{\pmb{G}}$ such that $U$ has type $\sigma$.
 Let $\varphi : \sigma\longrightarrow U$ be the strong isomorphism  from $\sigma$ to $U$,
 and let $u=\varphi(s)$.

 Suppose that $\pmb{F}$ is the finite  structure in $\mathcal{K}$ at the immediate successors of $s$ in $\tau$.
 Let
  $\pmb{\overline{F}}\in \mathcal{K}$
  be the structure
  containing a copy of $\pmb{F}$  satisfying the properties in Lemma \ref{lem.6.1}, \ref{lem.6.4}, or \ref{lem.6.7}, respectively.
   Since $\pmb{G}$ is a universal  inverse limit structure for $\mathcal{K}$, there is a copy
   $\pmb{\overline{H}}$ of $\pmb{\overline{F}}$  in the open interval $N_u$.
Taking  $t=\mbox{stem}(T_{\pmb{\overline{H}}})$,
then  succ$_{T_{\pmb{\overline{H}}}}(t)$ contains a copy of $\pmb{\overline{F}}$,
and thus succ$_{T_{\pmb{\overline{H}}}}(t)$
contains a copy $\pmb{H}$  of $\pmb{F}$.
Let $X$ denote the set of nodes  in succ$_{T_{\pmb{\overline{H}}}}(t)$ forming the universe of
$\pmb{H}$.
Let $Y$ be a set of nodes $\{y_z:z\in U\setminus \{u\}\}$  of length $|t|+1$ such that for each  $z\in U\setminus \{u\}$,
$y_z\supseteq z$.
Then $U\cup \{t\}\cup X\cup Y$ is
 a copy of $\tau$.
Hence, $\tau$ persists in $\pmb{G}$.

Now
given  a structure $\pmb{H}\in\mathcal{K}$,
 let $\tau_0,\dots,\tau_{m-1}$ list the collection of all types
 for copies of $\pmb{H}$  in $\pmb{F_{\max}}$.
Let $c :\big(\mathop{}_{ \ \ \pmb{H}}^{\pmb{F_{\max}}}\big)\longrightarrow m$
be defined by $c(\pmb{J})=i$ if and only if $T_{\pmb{J}}$ has type $\tau_i$.
Note that $c$ is in fact continuous, hence Baire-measurable.
By the above argument, there is a substructure $\pmb{G}$ of $\pmb{F_{\max}}$ which is again a universal inverse limit structure with the property that for each $i<m$, there is a copy of $\pmb{H}$ in $\pmb{G}$ with type $\tau_i$.
Therefore,  all colors $i<m$ persist in $\pmb{G}$.
Therefore, $T(\pmb{H},\pmb{F_{\max}})\ge m$.

By Theorem \ref{thm.5.9},  we know that $T(\pmb{H},\pmb{F_{\max}})\le m$ for finite Baire-measurable colorings.\
Therefore, $T(\pmb{H},\pmb{F_{\max}})$ is exactly the number of types  associated to $\pmb{H}$.
\end{proof}

\begin{remark}
Theorem \ref{thm.main} characterizes the exact big Ramsey degrees
under the finite Baire-measurable colorings
for some ordered structures with one (non-order) binary relation.
It seems likely that similar methods can be developed
to characterize the exact big Ramsey degrees for all structures considered in this paper in terms of the number of types.
\end{remark}

\vspace{0.2cm}

{\bf Acknowledgements}\ \  The authors would like to thank the anonymous referee for  carefully checking the original draft and giving us many helpful
suggestions for improvements.

\bibliographystyle{amsplain}

\begin{thebibliography}{10}


\bibitem{AH78} F. G. Abramson, L. Harrington, {\em Models without indiscernibles}, Journal of Symbolic Logic 43(3) (1978) 572--600.


\bibitem{Balkoproofs}
M. Balko, D. Chodounsk{\'{y}}, J.
 Hubi{\v{c}}ka, M.  Kone{\v{c}}n{\'{y}}, L. Vena,
{\it Big Ramsey degrees of 3-uniform hypergraphs are finite} (2020) 10 pp, submitted,
arXiv:2008.00268.


\bibitem{Balko7}
M. Balko, D. Chodounsk{\'{y}}, N. Dobrinen, J.
 Hubi{\v{c}}ka, M.  Kone{\v{c}}n{\'{y}}, L. Vena,   A. Zucker,
{\it Exact big {R}amsey degrees via coding trees},
(2021) 97 pp,
submitted, arXiv:2110.08409.


\bibitem{AB81} A. Blass,  {\em A partition theorem for perfect sets}, Proceedings of the American Mathematical Society 82(2) (1981) 271--277.

\bibitem{CDP}
R. Coulson, N. Dobrinen,  R. Patel,
{\em Fra\"{\i}ss\'{e} classes with simply characterized big Ramsey degrees},
(2020) 69 pp,   submitted, arXiv:2010.02034.


\bibitem{DD79} D. Devlin,   {\em Some partition theorems and ultrafilters on $\omega$}, PhD thesis, Dartmouth College
(1979).



\bibitem{ND17}  N. Dobrinen,   {\em The Ramsey Theory of the universal homogeneous triangle-free graph},
 Journal of Mathematical Logic,
 20(2) (2020), 2050012, 75 pp.


\bibitem{ND19}  N. Dobrinen,   {\em The Ramsey Theory of Henson graphs},
Journal of Mathematical Logic,
  to appear, arXiv:1901.06660.



\bibitem{DobIfCoLog} N. Dobrinen,
{\em Ramsey theory on infinite structures and the method of strong coding trees}, Contemporary Logic and Computing, College Publications, 2020.

\bibitem{FG68}  F. Galvin,   {\em Partition theorems for the real line}, Notices of The American Mathematical Society, 15 (1968) 660.

\bibitem{SG13}   S. Geschke,    {\em Clopen graphs}, Fundamenta Mathematicae, 220 (2013) 155--189.

\bibitem{JH66} J. D. Halpern, H. L\"{a}uchli,   {\em A partition theorem}, Transactions of the American Mathematical Society, 124 (1966) 360--367.

\bibitem{SH19} S. Huber, S. Geschke, M. Kojman,    {\em Partitioning subgraphs of  profinite ordered graphs}, Combinatorica, 39(3) (2019) 659--678.


\bibitem{Hubicka20}
J. Hubi{\v{c}}ka,
 {\em Big {R}amsey degrees using parameter spaces},
 (2020) 19 pp, submitted,
arXiv:2009.00967.


\bibitem{AK05} A. Kechris, V. Pestov, S. Todor\v{c}evi\'{c}, {\em  Fra\"{\i}ss\'{e} limits, Ramsey theory,
and topological dynamics of automorphism groups}, Geometric and Functional Analysis, 15(1)
(2005) 106--189.


\bibitem{CL06}C. Laflamme, N. W. Sauer,  V. Vuksanovic, {\em Canonical partitions of universal
structures}, Combinatorica, 26(2) (2006) 183--205.



\bibitem{CL10} C. Laflamme, L. Nguyen Van Th\'{e}, N. Sauer, {\em Partition properties of the
dense local order and a colored version of Milliken's theorem}. Combinatorica, 30 (2010) 83--104.


\bibitem{MasulovicFBRD18}
D. Ma\v{s}ulovi\'{c},
{\em Finite big Ramsey degrees in universal structures},
 Journal of  Combinatorial Theory Series A, 170 (2020), 105137, 30 pp.


\bibitem{MasulovicRDBVS20}
D. Ma\v{s}ulovi\'{c}, {\em Ramsey degrees: big v. small},
European Journal of Combinatorics, 95 (2021), Paper No.\ 103323, 25 pp.


\bibitem{KR81} K. R. Milliken,  {\em A partition theorem for the infinite subtrees of a tree}, Transactions of the American
Mathematical Society, 263(1) (1981)
137--148.



\bibitem{JV77} J. Ne\v{s}et\v{r}il, V. R\"{o}dl,  {\em Partitions of finite relational and set systems}, Journal of Combinatorial Theory Series A,  22(3) (1977) 289--312.


\bibitem{JV83} J. Ne\v{s}et\v{r}il, V. R\"{o}dl,  {\em Ramsey classes of set systems}, Journal of Combinatorial Theory Series
A, 34(2) (1983) 183--201.



\bibitem{NR84} J. Ne\v{s}et\v{r}il, V. R\"{o}dl,  {\em Combinatorial partitions of finite posets and lattices--Ramsey  lattices}, Algebra Universalis, 19(1) (1984) 106--119.




\bibitem{JN95} J. Ne\v{s}et\v{r}il,  {\em Ramsey theory}. In: R. L. Graham, M. Gr\"{o}tschel and L.
Lov\'{a}sz, eds, Handbook of Combinatorics, Vol. 2, 1331-1403, MIT Press,
Cambridge, MA, USA, 1995.


\bibitem{LN08} L. Nguyen Van Th\'{e}.  {\em Big Ramsey degrees and divisibility in classes of
ultrametric spaces}. Canadian Mathematical Bulletin, 51 (2008) 412--423.


\bibitem{NVT}  L. Nguyen Van Th\'{e}.  {\em
Structural {R}amsey theory with the {K}echris-{P}estov-{T}odorcevic correspondence in mind},
Habilitation, Universit{\'{e}} d'Aix-Marseille (2013), 44 pp.



\bibitem{PTW85} M.  Paoli, W. T. Trotter, and J. W. Walker, {\em Graphs and orders in Ramsey theory and in dimension theory}, In Ivan Rival, editor, {\em Graphs and Order,}, volume 147 of NATO AST series, pages 351--394.  Springer, 1985.



\bibitem{NS06}  N. W. Sauer,    {\em Coloring subgraphs of the Rado graph}, Combinatorica, 26(2) (2006) 231--253.


\bibitem{ST10} S. Todor\v{c}evi\'{c},   {\em Introduction to Ramsey spaces}, volume 174 of Annals of Mathematics Studies, Princeton University Press, Princeton, 2010.



\bibitem{YY18} Y. Y. Zheng,   {\em A collection of topological Ramsey spaces of trees
and their application to profinite graph theory}, Archive for Mathematical Logic, 57 (2018) 939--952.


	\bibitem{AZ19}  A. Zucker,   {\em Big Ramsey degrees and topological dynamics}, Groups, Geometry and Dynamics,
13(1) (2019) 235--276.


\bibitem{ZuckerForb}  A. Zucker,
{\em A note on big {R}amsey degrees}, (2020) 21 pp, submitted, arXiv:2004.13162.


\end{thebibliography}

\end{document}